\documentclass[11pt]{amsart}

\usepackage{enumerate, amsmath, amsthm, amsfonts, amssymb, xy,  mathrsfs, graphicx, paralist}
\usepackage[usenames, dvipsnames]{xcolor}
\usepackage[margin=1in]{geometry} 
\usepackage[bookmarks, bookmarksdepth=2, colorlinks=true, linkcolor=blue, citecolor=blue, urlcolor=blue]{hyperref}

\usepackage[boxsize=1em]{ytableau}

\usepackage{tikz}

\input xy
\xyoption{all}

\newtheorem{theorem}{Theorem}
\numberwithin{theorem}{section}
\numberwithin{equation}{section}

\newtheorem{proposition}[theorem]{Proposition}
\newtheorem{lemma}[theorem]{Lemma}
\newtheorem{corollary}[theorem]{Corollary}

\theoremstyle{definition}
\newtheorem{rmk}[theorem]{Remark}
\newenvironment{remark}[1][]{\begin{rmk}[#1] \pushQED{\qed}}{\popQED \end{rmk}}
\newtheorem{eg}[theorem]{Example}
\newenvironment{example}[1][]{\begin{eg}[#1] \pushQED{\qed}}{\popQED \end{eg}}
\newtheorem{defn}[theorem]{Definition}

\newcommand{\rB}{\mathrm{B}}

\newcommand{\bC}{\mathbf{C}}

\newcommand{\rC}{\mathrm{C}}

\newcommand{\rD}{\mathrm{D}}

\newcommand{\rE}{\mathrm{E}}

\newcommand{\cH}{\mathcal{H}}

\newcommand{\rH}{\mathrm{H}}

\newcommand{\rK}{\mathrm{K}}

\newcommand{\cL}{\mathcal{L}}

\newcommand{\rL}{\mathrm{L}}

\newcommand{\cM}{\mathcal{M}}

\newcommand{\bO}{\mathbf{O}}
\newcommand{\cO}{\mathcal{O}}

\newcommand{\bP}{\mathbf{P}}

\newcommand{\cQ}{\mathcal{Q}}

\newcommand{\cR}{\mathcal{R}}

\newcommand{\rR}{\mathrm{R}}

\newcommand{\bS}{\mathbf{S}}

\newcommand{\fS}{\mathfrak{S}}

\newcommand{\cU}{\mathcal{U}}

\newcommand{\cV}{\mathcal{V}}

\newcommand{\cW}{\mathcal{W}}

\newcommand{\bZ}{\mathbf{Z}}

\newcommand{\fg}{\mathfrak{g}}

\renewcommand{\phi}{\varphi}
\renewcommand{\emptyset}{\varnothing}
\newcommand{\eps}{\varepsilon}

\newcommand{\ol}[1]{\overline{#1}}

\newcommand{\arxiv}[1]{\href{http://arxiv.org/abs/#1}{{\tt arXiv:#1}}}

\makeatletter
\def\Ddots{\mathinner{\mkern1mu\raise\p@
\vbox{\kern7\p@\hbox{.}}\mkern2mu
\raise4\p@\hbox{.}\mkern2mu\raise7\p@\hbox{.}\mkern1mu}}
\makeatother

\DeclareMathOperator{\rank}{rank}

\DeclareMathOperator{\Sym}{Sym}

\DeclareMathOperator{\Aut}{Aut}

\DeclareMathOperator{\Tor}{Tor}

\DeclareMathOperator{\Spec}{Spec}

\DeclareMathOperator{\sgn}{sgn}

\newcommand{\GL}{\mathbf{GL}}

\newcommand{\Sp}{\mathbf{Sp}}
\newcommand{\SO}{\mathbf{SO}}

\newcommand{\IGr}{\mathbf{IGr}}

\title{Homology of Littlewood complexes}
\author{Steven V Sam}
\address{Department of Mathematics, University of California, Berkeley, CA}
\email{svs@math.berkeley.edu}

\author{Andrew Snowden}
\address{Department of Mathematics, MIT, Cambridge, MA} \email{asnowden@math.mit.edu}

\author{Jerzy Weyman}
\address{Department of Mathematics, Northeastern University, Boston, MA}
\email{j.weyman@neu.edu}

\date{February 8, 2013}


\thanks{S.~Sam was supported by an NDSEG fellowship and a Miller research fellowship.  A.~Snowden was partially supported by NSF fellowship DMS-0902661.  J.~Weyman was partially supported by NSF grant DMS-0901185.}

\newcommand{\lw}{{\textstyle \bigwedge}}
\DeclareMathOperator{\Rep}{Rep}
\DeclareMathOperator{\Hom}{Hom}

\newcommand{\Step}[1]{\removelastskip\vskip.5\baselineskip\noindent{\it Step #1{}.}}

\begin{document}

\maketitle

\begin{abstract}
Let $V$ be a symplectic vector space of dimension $2n$.  Given a partition $\lambda$ with at most $n$ parts, there is an associated irreducible representation $\bS_{[\lambda]}(V)$ of $\Sp(V)$.  This representation admits a resolution by a natural complex $L^{\lambda}_{\bullet}$, which we call the {\bf Littlewood complex}, whose terms are restrictions of representations of $\GL(V)$.  When $\lambda$ has more than $n$ parts, the representation $\bS_{[\lambda]}(V)$ is not defined, but the Littlewood complex $L^{\lambda}_{\bullet}$ still makes sense.  The purpose of this paper is to compute its homology.  We find that either $L^{\lambda}_{\bullet}$ is acyclic or it has a unique non-zero homology group, which forms an irreducible representation of $\Sp(V)$.  The non-zero homology group, if it exists, can be computed by a rule reminiscent of that occurring in the Borel--Weil--Bott theorem. This result can be interpreted as the computation of the ``derived specialization'' of irreducible representations of $\Sp(\infty)$, and as such categorifies earlier results of Koike--Terada on universal character rings.  We prove analogous results for orthogonal and general linear groups. Along the way, we will see two topics from commutative algebra: the minimal free resolutions of determinantal ideals and Koszul homology.
\end{abstract}

\tableofcontents

\section{Introduction}

\subsection{Statement of main theorem}

Let $V$ be a symplectic vector space over the complex numbers of dimension $2n$.  Associated to a partition $\lambda$ with at most $n$ parts there is an irreducible representation $\bS_{[\lambda]}(V)$ of $\Sp(V)$, and all irreducible representations of $\Sp(V)$ are uniquely of this form.  The space $\bS_{[\lambda]}(V)$ can be defined as the quotient of the usual Schur functor $\bS_{\lambda}(V)$ by the sum of the images of all of the ``obvious'' $\Sp(V)$-linear maps $\bS_{\mu}(V) \to \bS_{\lambda}(V)$, where $\mu$ can be obtained by removing from $\lambda$ a vertical strip of size two.  In other words, we have a presentation
\begin{displaymath}
\bigoplus_{\lambda/\mu=(1,1)} \bS_{\mu}(V) \to \bS_{\lambda}(V) \to \bS_{[\lambda]}(V) \to 0.
\end{displaymath}
This presentation admits a natural continuation to a resolution $L^{\lambda}_{\bullet}=L^{\lambda}_{\bullet}(V)$, which we call the {\bf Littlewood complex}.  It can be characterized as the minimal resolution of $\bS_{[\lambda]}(V)$ by representations which extend to $\GL(V)$.

When the number of parts of $\lambda$ exceeds $n$, it still makes sense to speak of the complex $L^{\lambda}_{\bullet}$, even though there is no longer an associated irreducible $\bS_{[\lambda]}(V)$ (see \S \ref{ss:introex} for a simple example).  However, $L^{\lambda}_{\bullet}$ is typically no longer exact in higher degrees.  A very natural problem is to compute its homology, and this is exactly what the main theorem of this paper accomplishes:

\begin{theorem}
\label{mainthm}
The homology of $L^{\lambda}_{\bullet}$ is either identically zero or else there exists a unique $i$ for which $\rH_i(L^{\lambda}_{\bullet})$ is non-zero, and it is then an irreducible representation of $\Sp(V)$.
\end{theorem}

In fact, there is a procedure, called the {\bf modification rule} (see \S \ref{s:modrule}), which allows one to compute exactly which homology group is non-zero and which irreducible representation it is.  This rule can be phrased in terms of a certain Weyl group action, and, in this way, the theorem is reminiscent of the classical Borel--Weil--Bott theorem.   There is also a more combinatorial description of the rule in terms of border strips.  See Theorem~\ref{s:mainthm} for a precise statement.

We prove analogous theorems for the orthogonal and general linear groups, but for clarity of exposition we concentrate on the symplectic case in the introduction.  An analogous result for the symmetric group can be found in \cite[Proposition 7.4.3]{tca}; this will be elaborated upon in \cite{algrep}.

\subsection{An example}
\label{ss:introex}

Let us now give the simplest example of the theorem, namely $\lambda=(1,1)$.  If $n \ge 2$ then the irreducible representation $\bS_{[1,1]}(V)$ is the quotient of $\bS_{(1,1)}(V)=\lw^2{V}$ by the line spanned by the symplectic form (where we identify $V$ with $V^*$ via the form).  The complex $L_\bullet^{\lambda}$ is thus
\begin{displaymath}
\cdots \to 0 \to \bC \to \lw^2{V},
\end{displaymath}
the differential being multiplication by the form.  This complex clearly makes since even if $n<2$.  When $n=0$, the differential is surjective, and $\rH_1=\bC$, the trivial representation of $\Sp(V)$.  When $n=1$, the differential is an isomorphism and all homology vanishes.  And when $n \ge 2$, the differential is injective and $\rH_0=\bS_{[1,1]}(V)$.  More involved examples can be found in \S \ref{ss:examples}.

\subsection{Representation theory of $\Sp(\infty)$}

The proper context for Theorem~\ref{mainthm} lies in the representation theory of $\Sp(\infty)$.  We now explain the connection, noting, however, that the somewhat exotic objects discussed here are not used in our proof of Theorem~\ref{mainthm} and do not occur in the remainder of the paper.  Let $\Rep(\Sp(\infty))$ denote the category of ``algebraic'' representations\footnote{Technically, we should use the ``pro'' version of the category, which is opposite to the more usual ``ind'' version of the category.  See \cite{algrep} for details.} of $\Sp(\infty)$.  This category was first identified in \cite{categorytg}, where it is denoted $\mathbb{T}_{\fg}$.  It is also studied from a slightly different point of view in \cite{algrep}. As shown in \cite{algrep}, there is a specialization functor
\begin{displaymath}
\Gamma_V \colon \Rep(\Sp(\infty)) \to \Rep(\Sp(V)).
\end{displaymath}
This functor is right exact, but not exact --- the category $\Rep(\Sp(\infty))$ is not semi-simple.  The Littlewood complex $L^{\lambda}_\bullet(\bC^{\infty})$ makes sense, and defines a complex in $\Rep(\Sp(\infty))$.  It is exact in positive degrees and its $\rH_0$ is a simple object $\bS_{[\lambda]}(\bC^{\infty})$; all simple objects are uniquely of this form.  The Schur functor $\bS_{\lambda}(\bC^{\infty})$, as an object of $\Rep(\Sp(\infty))$, has two important properties:  it is projective and it specializes under $\Gamma_V$ to $\bS_{\lambda}(V)$.  We thus see that $L^{\lambda}_\bullet(\bC^{\infty})$ is a projective resolution of $\bS_{[\lambda]}(\bC^{\infty})$ and specializes under $\Gamma_V$ to $L^{\lambda}_\bullet(V)$.  We therefore have the following observation, which explains the significance of the Littlewood complex from this point of view:

\begin{proposition}
We have $L_\bullet^{\lambda}(V)=\rL\Gamma_V(\bS_{[\lambda]}(\bC^{\infty}))$, i.e., $L_\bullet^{\lambda}(V)$ computes the derived specialization of the simple object $\bS_{[\lambda]}(\bC^{\infty})$ to $V$.
\end{proposition}

We can thus rephrase Theorem~\ref{mainthm} as follows:

\begin{theorem}
\label{mainthm2}
Let $M$ be an irreducible algebraic representation of $\Sp(\infty)$.  Then $\rL\Gamma_V(M)$ is either acyclic or else there is a unique $i$ for which $\rL^i\Gamma_V(M)$ is non-zero, and it is then an irreducible representation of $\Sp(V)$.
\end{theorem}

The symplectic Schur functors $\bS_{[\lambda]}$ exhibit stability for large dimensional vector spaces (as explained in \cite{koiketerada}, but see also \cite{ew} and \cite{htw}), but not in general, in contrast to the usual Schur functors.  A general strategy for dealing with problems involving symplectic Schur functors is to pass to the stable range (e.g., work with $\bC^{\infty}$), take advantage of the simpler behavior there, and then apply the specialization functor to return to the unstable range.  For this strategy to be viable, one must understand the behavior of the specialization functor.  This was one source of motivation for this project, and is accomplished by Theorem~\ref{mainthm2}.

\subsection{Relation to results of Koike--Terada}

Theorem~\ref{mainthm2} categorifies results of \cite{koiketerada}, as we now explain.  In \cite{koiketerada}, a so-called universal character ring $\Lambda$ is defined, and a ring homomorphism $\pi$ (``specialization'') from $\Lambda$ to the representation ring of $\Sp(V)$ is given.  A basis $s_{[\lambda]}$ of $\Lambda$ is given and it is shown that the image under $\pi$ of $s_{[\lambda]}$ is either 0 or (plus or minus) the character of an irreducible representation of $\Sp(V)$.  In fact, $\Lambda$ is the Grothendieck group of $\Rep(\Sp(\infty))$, $\pi$ is the map induced by the specialization functor $\Gamma_V$ and $s_{[\lambda]}$ is the class of the simple object $\bS_{[\lambda]}(\bC^{\infty})$ in the Grothendieck group.  Thus the $\rK$-theoretic shadow of Theorem~\ref{mainthm2} is precisely the result of \cite{koiketerada} on specialization. However, we note that our proof depends on \cite{koiketerada}.

\subsection{Koszul homology and classical invariant theory} \label{ss:intro:koszul}

Theorem~\ref{mainthm} can be reinterpreted as the calculation of the homology groups of the Koszul complex on the generators of an ideal which arises in classical invariant theory. This will be explained in \S\ref{ss:ccx} (see also \S\ref{ss:bdcx} and \S\ref{ss:glcx} for the orthogonal and general linear groups). For now, we remark that Koszul homology seems to be remarkably difficult to calculate, even for well-behaved classes of ideals, such as determinantal ideals. Very few cases have been worked out explicitly; we point to \cite{ah} for the case of codimension 2 perfect ideals, and \cite{koszulhomology} for the case of codimension 3 Gorenstein ideals. Both of these classes of ideals are determinantal. They are singled out because their Koszul homology modules are Cohen--Macaulay (this property fails for all other determinantal ideals).

\subsection{Resolutions of determinantal ideals}

In \S\ref{ss:commalg}, we will see how the interpretation of Theorem~\ref{mainthm} in terms of Koszul homology in \S\ref{ss:intro:koszul} can also be interpreted in terms of the minimal free resolutions of certain modules $M_{\lambda}$ supported on the determinantal varieties defined by the Pfaffians of a generic skew-symmetric matrix.  The coordinate ring of the determinantal variety is the module $M_\emptyset$ and so the computation of its resolution becomes a special case of Theorem~\ref{mainthm}, and therefore realizes this classical resolution as the first piece of a much larger structure.  The orthogonal group and general linear group correspond to determinantal varieties in generic symmetric matrices and generic matrices, respectively. We refer the reader to \cite[\S 6]{weyman} for the calculation of the minimal free resolutions of the  coordinate rings of determinantal varieties.

\subsection{Overview of proof} \label{ss:proofoverview}

There are three main steps to the proof:
\begin{enumerate}[(a)]
\item We first establish a combinatorial result, relating Bott's algorithm for calculating cohomology of irreducible homogeneous bundles to the modification rule appearing in the main theorem.
\item We then introduce a certain module $M_{\lambda}$ over the polynomial ring $A=\Sym(\lw^2{E})$ (where $E$ is an auxiliary vector space), and compute its minimal free resolution.  The main tools are step (a), the Borel--Weil--Bott theorem and the geometric method of the third author.
\item Lastly, we identify $M_{\lambda}$ with the $\bS_{[\lambda]}(V)$-isotypic piece of the ring $B=\Sym(V \otimes E)$.  The results of step (b) and the specialization homomorphism on K-theory (see \cite{koike}, \cite{koiketerada}, \cite{wenzl}) are used to get enough control on $M_{\lambda}$ to do this.  Once the identification is made, the results of step (b) give the minimal free resolution of $B$ as an $A$-module.
\end{enumerate}
The theorem then follows, as the Littlewood complex can be identified with a piece of the minimal free resolution of $B$ over $A$.

\subsection{Notation and conventions}

We always work over the complex numbers. It is possible to work over any field of characteristic 0, but there does not seem to be any advantage to doing so. We write $\ell(\lambda)$ for the number of parts of a partition $\lambda$.  The rank of a partition $\lambda$, denoted $\rank(\lambda)$, is the number of boxes on the main diagonal.  We write $\lambda^{\dag}$ for the transpose of the partition $\lambda$.  We will occasionally use Frobenius coordinates to describe partitions, which we now recall.  Let $r=\rank(\lambda)$.  For $1 \le i \le r$, let $a_i$ (resp.\ $b_i$) denote the number of boxes to the right (resp.\ below) the $i$th diagonal box, including the box itself.  Then the Frobenius coordinates of $\lambda$ are $(a_1, \ldots, a_r \vert b_1, \ldots, b_r)$.  We denote by $c^{\lambda}_{\mu,\nu}$ the Littlewood--Richardson coefficients, i.e., the coefficient of the Schur function $s_\lambda$ in the product $s_\mu s_\nu$. For the relevant background on partitions, Schur functions, and Schur functors, we refer to \cite[Chapter 1]{macdonald} and \cite[Chapters 1, 2]{weyman}.

\section{Preliminaries}

\subsection{The geometric technique}

Let $X$ be a smooth projective variety.  Let
\begin{displaymath}
0 \to \xi \to \eps \to \eta \to 0
\end{displaymath}
be an exact sequence of vector bundles on $X$, with $\eps$ trivial, and let $\cV$ be another vector bundle on $X$.  Put
\begin{displaymath}
A=\rH^0(X, \Sym(\eps)), \qquad M=\rH^0(X, \Sym(\eta) \otimes \cV).
\end{displaymath}
Then $A$ is a ring --- in fact, it is the symmetric algebra on $\rH^0(X, \eps)$ --- and $M$ is an $A$-module.  The following proposition encapsulates what we need of the geometric technique of the third author.  For a proof, and a stronger result, see \cite[\S 5.1]{weyman}.

\begin{proposition}
\label{geo}
Assume $\rH^j(X, \lw^{i+j}(\xi) \otimes \cV)=0$ for $i<0$ and all $j$.  Then we have a natural isomorphism
\begin{displaymath}
\Tor^A_i(M, \bC) = \bigoplus_{j \ge 0} \rH^j(X, \lw^{i+j}(\xi) \otimes \cV).
\end{displaymath}
\end{proposition}

\subsection{The Borel--Weil--Bott theorem}
\label{ss:bott}

Let $\cU$ be the set of all integer sequences $(a_1, a_2, \ldots)$ which are eventually 0.  We identify partitions with non-increasing sequences in $\cU$ (such sequences are necessarily non-negative).  For $i \ge 1$, let $s_i$ be the transposition which switches $a_i$ and $a_{i+1}$, and let $\fS$ be the group of automorphisms of $\cU$ generated by the $s_i$.  The group $\fS$ is a Coxeter group (in fact, the infinite symmetric group), and admits a length function $\ell \colon \fS \to \bZ_{\ge 0}$. By definition, the length of $w \in \fS$ is the minimum number $\ell(w)$ so that there exists an expression 
\begin{align} \label{eqn:word}
w = s_{i_1} \cdots s_{i_\ell(w)}.
\end{align} Alternatively, $\ell(w)$ is the number of inversions of $w$, interpreted as a permutation.

We define a second action of $\fS$ on $\cU$, denoted $\bullet$, as follows.  For $w \in \fS$ and $\lambda \in \cU$ we put $w \bullet \lambda=w(\lambda+\rho)-\rho$, where $\rho=(-1, -2, \ldots)$.  In terms of the generators, this action is:
\begin{displaymath}
s_i \bullet (\ldots, a_i, a_{i+1}, \ldots) = (\ldots, a_{i+1}-1, a_i+1, \ldots).
\end{displaymath}
Let $\lambda$ be an element of $\cU$.  Precisely one of the following two possibilities occurs:
\begin{itemize}
\item There exists a unique element $w$ of $\fS$ such that $w \bullet \lambda$ is a partition.  In this case, we call $\lambda$ {\bf regular}.
\item There exists an element $w \ne 1$ of $\fS$ such that $w \bullet \lambda=\lambda$.   In this case, we call $\lambda$ {\bf singular}.
\end{itemize}
{\bf Bott's algorithm} \cite[\S 4.1]{weyman} is a procedure for determining if $\lambda$ is regular.  It goes as follows.  Find an index $i$ such that $\lambda_{i+1}>\lambda_i$.  If no such index exists, then $\lambda$ is a partition and is regular.  If $\lambda_{i+1}-\lambda_i=1$ then $\lambda$ is singular.  Otherwise apply $s_i$ to $\lambda$ and repeat.  Keeping track of the $s_i$ produces a minimal word for the element $w$ in the definition \eqref{eqn:word}. In particular, it is important to note that we have a choice of which index $i$ to pick in the first step. Different choices lead to different minimal words, but the resulting partition and permutation are independent of these choices.

Let $E$ be a vector space and let $X$ be the Grassmannian of rank $n$ quotients of $E$.  (We assume $\dim{E} \ge n$, obviously.)  We have a tautological sequence on $X$
\begin{align} \label{eqn:taut}
0 \to \cR \to E \otimes \cO_X \to \cQ \to 0,
\end{align}
where $\cQ$ has rank $n$.  For a partition $\lambda$ with at most $n$ parts and a partition $\mu$, let $(\lambda \mid_n \mu)$ be the element of $\cU$ given by $(\lambda_1, \ldots, \lambda_n, \mu_1, \mu_2, \ldots)$.  The Borel--Weil--Bott theorem \cite[\S 4.1]{weyman} is then:

\begin{theorem}[Borel--Weil--Bott]
Let $\lambda$ be a partition with at most $n$ parts, let $\mu$ be any partition and let $\cV$ be the vector bundle $\bS_{\lambda}(\cQ) \otimes \bS_{\mu}(\cR)$ on $X$.
\begin{itemize}
\item Suppose $(\lambda \mid_n \mu)$ is regular, and write $w \bullet (\lambda \mid_n \mu)=\alpha$ for a partition $\alpha$.  Then
\begin{displaymath}
\rH^i(X, \cV) = \begin{cases}
\bS_{\alpha}(E) & \textrm{if $i=\ell(w)$} \\
0 & \textrm{otherwise.}
\end{cases}
\end{displaymath}
\item Suppose $(\lambda \mid_n \mu)$ is singular.  Then $\rH^i(X, \cV)=0$ for all $i$.
\end{itemize}
\end{theorem}

\begin{remark}
If $\ell(\mu)>\rank(\cR)$ then $\cV=0$.  Similarly, if $\ell(\alpha)>\dim(E)$ then $\bS_{\alpha}(E)=0$.  These problems disappear if $\dim(E)$ is sufficiently large compared to $\lambda$ and $\mu$; in fact, the situation becomes completely uniform when $\dim(E)=\infty$.
\end{remark}

\subsection{Resolution of the second Veronese ring}

In our treatment of odd orthogonal groups, we need to know the resolution for the second Veronese ring in a relative setting.  We state the relevant result here, so as not to interrupt the discussion later.

Let $X$ be a variety and let $\cR$ be a vector bundle on $X$.  Let $\pi \colon \bP(\cR) \to X$ be the associated projective space bundle of one dimensional quotients of $\cR$.  Let $\pi^*(\cR) \to \cL$ be the universal rank one quotient.  Define $\xi$ to be the kernel of the map $\Sym^2(\pi^* \cR) \to \Sym^2(\cL)$.  The result we need is the following: 

\begin{proposition}
\label{prop:veronese}
Let $a$ be $0$ or $1$.  We have
\begin{displaymath}
\bigoplus_{j \in \bZ} \rR^j\pi_*(\lw^{i+j}(\xi) \otimes \cL^a)=\bigoplus_{\mu} \bS_{\mu}(\cR),
\end{displaymath}
where the sum is over those partitions $\mu$ such that $\mu=\mu^{\dag}$, $\rank(\mu)=a \pmod{2}$ and $i=\tfrac{1}{2}(\vert \mu \vert-\rank(\mu))$.
\end{proposition}

\begin{proof}
This is a relative version of the calculation of the minimal free resolution (over $\Sym(U)$) of the second Veronese ring $M(U)_0 = \bigoplus_{d \ge 0} \Sym^{2d}(U)$ ($a=0$) and its odd Veronese module $M(U)_1 = \bigoplus_{d \ge 0} \Sym^{2d+1}(U)$ ($a=1$), where $U$ is some vector space.

The case $a=0$ is contained in \cite[Theorem 6.3.1(c)]{weyman}. Now we calculate the case $a=1$. Note that it is functorial in $\cR$, so due to the stability properties of Schur functors, if we calculate the resolution for $\rank U = N$, the same result also holds for $\rank U < N$. So it is enough to handle the case that $\rank U$ is odd, but arbitrarily large. 

Consider the total space of $\cO(-2)$ on $\bP(U)$ with structure map $\pi' \colon \cO(-2) \to \bP(U)$, and define $\cL' = {\pi'}^*\cO_{\bP(U)}(1)$. Also consider the map $p \colon \cO(-2) \to \Spec(\Sym(U))$. Then $M(U)_a = p_*({\cL'}^{\otimes a})$. Using this setup and \cite[Theorem 5.1.4]{weyman}, we see that the Ext dual \cite[Proposition 1.2.5]{weyman} of $M(U)_0$ is $M(U)_1$ when $\rank U$ is odd. All of the partitions $\mu$ in the free resolution of $M(U)_0$ fit in a square of size $\rank U$, and on the level of the partitions that index the Schur functors appearing in the free resolution, this duality amounts to taking complements within this square, and then reversing the direction of the arrows, hence the result follows.
\end{proof}

\begin{remark}
Let $\eps=\Sym^2(\pi^* \cR)$ and $\eta=\Sym^2(\cL)$, so that we have an exact sequence
\begin{displaymath}
0 \to \xi \to \eps \to \eta \to 0.
\end{displaymath}
The ring $\pi_*(\Sym(\eta))$ is identified with $\Sym(\Sym^2(\cR))$, i.e., the projective coordinate ring of the second Veronese of $\bP(\cR)$.  By a relative version of the geometric method, its minimal locally free resolution is computed by $\rR\pi_*(\lw^{\bullet}(\xi))$, i.e., the sheaves appearing in the proposition.
\end{remark}

\subsection{A criterion for degeneration of certain spectral sequences}
\label{ss:degen}

Let $\pi \colon X' \to X$ be a map of proper varieties and let $\cV$ be a coherent sheaf on $X'$.  We say that $(\pi, \cV)$ is {\bf degenerate} if the Leray spectral sequence
\begin{displaymath}
\rE^{i,j}_2 = \rH^i(X, \rR^j\pi_*(\cV)) \Rightarrow \rH^{i+j}(X, \cV)
\end{displaymath}
degenerates at the second page.  The following is a simple criterion for degeneracy that applies in our one case of interest:

\begin{lemma}
\label{lem:degen}
Suppose that a group $G$ acts on $X$ and $X'$ and that $\pi$ and $\cV$ are $G$-equivariant.  Suppose furthermore that the $G$-module $\bigoplus_{i,j} \rH^i(X, \rR^j\pi_*(\cV))$ is semi-simple and multiplicity-free.  Then $(\pi, \cV)$ is degenerate.
\end{lemma}

\begin{proof}
The differentials of the spectral sequence are $G$-equivariant, and thus forced to vanish.
\end{proof}

\subsection{A lemma from commutative algebra} \label{ss:commalg}

We now give a very simple lemma that allows us to interpret Koszul homology groups as $\Tor$'s.  This is useful since we are ultimately interested in certain Koszul homology groups, but the geometric technique computes $\Tor$'s.

Let $B$ be a graded $\bC$-algebra and let $U$ be a homogeneous subspace of $B$.  We can then form the Koszul complex $\rK_{\bullet}=B \otimes \lw^{\bullet}{U}$.  If $f_1, \ldots, f_n$ is a basis for $U$ then $\rK_{\bullet}$ is the familiar Koszul complex on the $f_i$.  Let $A=\Sym(U)$, so that there is a natural homomorphism $A \to B$.  We then have the following result:

\begin{lemma}
\label{lem:koszul}
There is a natural identification $\Tor^A_i(B, \bC)=\rH_i(\rK_{\bullet})$.
\end{lemma}

\begin{proof}
We can resolve $\bC$ as an $A$-module using the Koszul resolution $A \otimes \lw^{\bullet}{U}$.  Tensoring over $A$ with $B$ gives $\rK_{\bullet}$, and is also how one computes $\Tor^A_{\bullet}(B, \bC)$.
\end{proof}

\section{Symplectic groups}

\subsection{Representations of $\Sp(V)$}

Let $(V, \omega)$ be a symplectic space of dimension $2n$ (here $\omega \in \bigwedge^2 V^*$ is the symplectic form, and gives an isomorphism $V \cong V^*$).  As stated in the introduction, the irreducible representations of $\Sp(V)$ are indexed by partitions $\lambda$ with $\ell(\lambda) \le n$ (see \cite[\S 17.3]{fultonharris}).  We call such partitions {\bf admissible}.  For an admissible partition $\lambda$, we write $\bS_{[\lambda]}(V)$ for the corresponding irreducible representation of $\Sp(V)$.

\subsection{The Littlewood complex}
\label{ss:ccx}

Let $E$ be a vector space.  Put $U=\lw^2{E}$, $A=\Sym(U)$ and $B=\Sym(E \otimes V)$.  Consider the inclusion $U \subset B$ given by
\begin{displaymath}
\lw^2{E} \subset \lw^2{E} \otimes \lw^2{V} \subset \Sym^2(E \otimes V),
\end{displaymath}
where the first inclusion is multiplication by $\omega$.  This inclusion defines an algebra homomorphism $A \to B$.  Put $C=B \otimes_A \bC$; this is the quotient of $B$ by the ideal generated by $U$.  We have maps
\begin{displaymath}
\Spec(C) \to \Spec(B) \to \Spec(A).
\end{displaymath}
We have a natural identification of $\Spec(B)$ with the space $\Hom(E, V)$ of linear maps $\varphi \colon E \to V$ and of $\Spec(A)$ with the space $\lw^2(E)^*$ of anti-symmetric forms on $E$.  The map $\Spec(B) \to \Spec(A)$ takes a linear map $\varphi$ to the form $\varphi^*(\omega)$.  The space $\Spec(C)$, which we call the {\bf Littlewood variety}, is the scheme-theoretic fiber of this map above 0, i.e., it consists of those maps $\varphi$ for which $\varphi^*(\omega)=0$. In other words, $\Spec(C)$ consists of maps $\phi \colon E \to V$ such that the image of $\phi$ is an isotropic subspace of $V$.

Let $\rK_{\bullet}(E)=B \otimes \lw^{\bullet}{U}$ be the Koszul complex of the Littlewood variety.  We can decompose this complex under the action of $\GL(E)$:
\begin{displaymath}
\rK_{\bullet}(E)=\bigoplus_{\ell(\lambda) \le \dim{E}} \bS_{\lambda}(E) \otimes L^{\lambda}_{\bullet}.
\end{displaymath}
The complex $L^{\lambda}_{\bullet}$ is the {\bf Littlewood complex}, and is independent of $E$ (so long as $\dim{E} \ge \ell(\lambda)$). By \cite[Theorem 3.8.6.2]{howe}, its zeroth homology is 
\begin{equation}
\label{s:Hzero}
\rH_0(L^{\lambda}_{\bullet})=\begin{cases}
\bS_{[\lambda]}(V) & \textrm{if $\lambda$ is admissible} \\
0 & \textrm{otherwise.}
\end{cases}
\end{equation}
By Lemma~\ref{lem:koszul}, we have $\rH_i(\rK_{\bullet})=\Tor^A_i(B, \bC)$, and so we have a decomposition
\begin{equation}
\label{s:tor}
\Tor^A_i(B, \bC)=\bigoplus_{\ell(\lambda) \le \dim{E}} \bS_{\lambda}(E) \otimes \rH_i(L^{\lambda}_{\bullet}).
\end{equation}
Applied to $i=0$, we obtain
\begin{equation}
\label{s:cring}
C=\bigoplus_{\textrm{admissible $\lambda$}} \bS_{\lambda}(E) \otimes \bS_{[\lambda]}(V).
\end{equation}

\begin{remark}
It is possible to compute the terms of $L^{\lambda}_{\bullet}$ explicitly.  Let $Q_{-1}$ be the set of partitions $\lambda$ whose Frobenius coordinates $(a_1, \ldots, a_r \vert b_1, \ldots, b_r)$ satisfy $a_i=b_i-1$ for each $i$ (see \S \ref{ss:sthm} for further discussion of this set).  Then
\begin{displaymath}
L^{\lambda}_i=\bigoplus_{\substack{\mu \in Q_{-1},\\ \vert \mu \vert=2i}} \bS_{\lambda/\mu}(V).
\end{displaymath}
If $\lambda$ is admissible then the higher homology of $L^{\lambda}_{\bullet}$ vanishes (see Proposition~\ref{s:spcase} below), and so, taking Euler characteristics, we get an equality in the representation ring of $\Sp(V)$:
\begin{displaymath}
[\bS_{[\lambda]}(V)]=\sum_{\mu \in Q_{-1}} (-1)^{\vert \mu \vert/2} [\bS_{\lambda/\mu}(V)].
\end{displaymath}
The significance of this identity is that it expresses the class of the irreducible $\bS_{[\lambda]}(V)$ in terms of representations which are restricted from $\GL(V)$.  It is due to Littlewood \cite[p.295]{littlewood} (see also \cite[Prop.~1.5.3(2)]{koiketerada}), and is why we name the complexes $L^{\lambda}_{\bullet}$ after him.  
\end{remark}

\subsection{A special case of the main theorem} \label{ss:C:special}

Our main theorem computes the homology of the complex $L^{\lambda}_{\bullet}$.  We now formulate and prove the theorem in a particularly simple case.  We mention this here only because it is worthwhile to know; the argument is not needed to prove the main theorem.

\begin{proposition}
\label{s:spcase}
Suppose $\lambda$ is admissible.  Then
\begin{displaymath}
\rH_i(L^{\lambda}_{\bullet}) = \begin{cases}
\bS_{[\lambda]}(V) & \textrm{if $i=0$} \\
0 & \textrm{otherwise.}
\end{cases}
\end{displaymath}
\end{proposition}

\begin{proof}
Choose $E$ to be of dimension $n$.  By Lemma~\ref{lem:C:CI} below, $\rK_{\bullet}(E)$ has no higher homology.  It follows that $L^{\lambda}_{\bullet}$ does not either.  The computation of $\rH_0(L^{\lambda}_{\bullet})$ is given in \eqref{s:Hzero}.
\end{proof}

\begin{lemma} \label{lem:C:CI}
Suppose $\dim{E} \le n$.  Then $U \subset B$ is spanned by a regular sequence.
\end{lemma}

\begin{proof}
It suffices to show that $\dim{\Spec(C)}=\dim{\Spec(B)}-\dim{U}$.  Put $d=\dim{E}$.  Observe that the locus in $\Spec(C)$ where $\varphi$ is injective is open. Let $\IGr(d,V)$ be the variety of $d$-dimensional isotropic subspaces of $V$, which comes with a rank $d$ tautological bundle $\cR \subset V \otimes \cO_{\IGr(d,V)}$. There is a natural birational map from the total space of $\cH om(E,\cR)$ to $\Spec(C)$, and thus $\Spec(C)$ has dimension $2nd-\tfrac{1}{2} d(d-1)$.  As $\dim{\Spec(B)}=2nd$ and $\dim{U}=\tfrac{1}{2} d(d-1)$, the result follows.
\end{proof}

\subsection{The modification rule}
\label{s:modrule}

We now associate to a partition $\lambda$ two quantities, $i_{2n}(\lambda)$ and $\tau_{2n}(\lambda)$, which will be used to describe the homology of $L_\bullet^{\lambda}$.  (Recall that $2n=\dim{V}$.)  We give two equivalent definitions of these quantities, one via a Weyl group action and one via border strips.

We begin with the Weyl group definition, following \cite[\S 1.5]{wenzl}.  Recall that in \S \ref{ss:bott} we defined automorphisms $s_i$ of the set $\cU$ of integer sequences, for $i \ge 1$.  We now define an additional automorphism:  $s_0$ negates $a_1$.  We let $W$ be the group generated by the $s_i$, for $i \ge 0$.  Then $W$ is a Coxeter group of type $\rB\rC_{\infty}$, and, as such, is equipped with a length function $\ell \colon W \to \bZ_{\ge 0}$, which is defined just as in \eqref{eqn:word}.  Let $\rho=( -(n+1), -(n+2), \ldots)$.  Define a new action of $W$ on $\cU$ by $w \bullet \lambda=w(\lambda+\rho)-\rho$.  On $\fS$ this action agrees with the one defined in \S \ref{ss:bott}, despite the difference in $\rho$.  The action of $s_0$ is given by
\begin{displaymath}
s_0 \bullet (a_1, a_2, \ldots)=(2n+2-a_1, a_2, \ldots).
\end{displaymath}
Given a partition $\lambda \in \cU$, exactly one of the following two possibilities hold:
\begin{itemize}
\item There exists a unique element $w \in W$ such that $w \bullet \lambda^{\dag}=\mu^{\dag}$ is a partition and $\mu$ is admissible.  We then put $i_{2n}(\lambda)=\ell(w)$ and $\tau_{2n}(\lambda)=\mu$.
\item There exists a non-identity element $w \in W$ such that $w \bullet \lambda^{\dag}=\lambda^{\dag}$.  We then put $i_{2n}(\lambda)=\infty$ and leave $\tau_{2n}(\lambda)$ undefined.
\end{itemize}
Note that if $\lambda$ is an admissible partition then we are in the first case with $w=1$, and so $i_{2n}(\lambda)=0$ and $\tau_{2n}(\lambda)=\lambda$.

~

We now give the border strip definition, following \cite[\S 5]{sundaram} (which is based on \cite{king}).  If $\ell(\lambda) \le n$ we put $i_{2n}(\lambda)=0$ and $\tau_{2n}(\lambda)=\lambda$.  Suppose $\ell(\lambda)>n$. Recall that a {\bf border strip} is a connected skew Young diagram containing no $2 \times 2$ square.  Let $R_{\lambda}$ be the connected border strip of length $2(\ell(\lambda)-n-1)$ which starts at the first box in the final row of $\lambda$, if it exists.  If $R_{\lambda}$ exists, is non-empty and $\lambda \setminus R_{\lambda}$ is a partition, then we put $i_{2n}(\lambda)=c(R_{\lambda})+i_{2n}(\lambda \setminus R_{\lambda})$ and $\tau_{2n}(\lambda)=\tau_{2n}(\lambda \setminus R_{\lambda})$, where $c(R_{\lambda})$ denotes the number of columns that $R_{\lambda}$ occupies; otherwise we put $i_{2n}(\lambda)=\infty$ and leave $\tau_{2n}(\lambda)$ undefined.

\begin{remark} \label{rmk:striphook}
There is an alternative way to think about removing $R_\lambda$ in terms of hooks. Given a box $b$ in the Young diagram of $\lambda$, recall that the book of $b$ is the set of boxes which are either directly below $b$ or directly to the right of $b$ (including $b$ itself). The border strips $R$ of $\lambda$ that begin at the last box in the first column, and have the property that $\lambda \setminus R$ is a Young diagram, are naturally in bijection with the boxes in the first column: just take the box $b_R$ in the same row where $R$ ends. The important point is that the size of this border strip is the same as size of the hook of $b_R$, and removing $R$ is the same as removing the hook of $b_R$ and shifting all boxes below this hook one box in the northwest direction. This is illustrated in the following diagram:
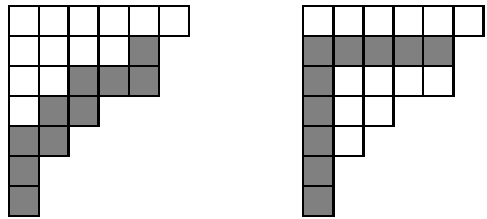
\begin{displaymath}
\ydiagram[*(white)]{6,4,2,1}*[*(gray)]{6,5,5,3,2,1,1} \qquad\qquad
\ydiagram[*(white)]{6}*[*(gray)]{6,5}*[*(white)]{6,5,1+4,1+2,1+1}*[*(gray)]{6,5,5,3,2,1,1}
\end{displaymath}
The shaded boxes indicate the border strip (left diagram) and hook (right diagram).
\end{remark}

The agreement of the above two definitions may be known to some experts, but we are unaware of a reference, so we provide a proof.

\begin{proposition}
The above two definitions agree.
\end{proposition}

\begin{proof}
Suppose that we are removing a border strip $R_\lambda$ of size $2(\ell(\lambda) - n - 1)$ from $\lambda$ which begins at the first box in the final row of $\lambda$. Let $c = c(R_\lambda)$ be the number of columns of $R_\lambda$. The sequence $(s_{c - 1} s_{c - 2} \cdots s_1 s_0) \bullet \lambda^\dagger$ is
\[
(\lambda^\dagger_2 - 1, \lambda^\dagger_3 - 1, \dots, \lambda^\dagger_c - 1, 2n+2 - \lambda^\dagger_1 + c - 1, \lambda^\dagger_{c+1}, \lambda^\dagger_{c+2}, \dots),
\]
and these are the same as the column lengths of $\lambda \setminus R_\lambda$.

Conversely, if we use the Weyl group modification rule with $w \in W$, then the expression \eqref{eqn:word} for $w$ must begin with $s_0$: if we apply any $s_i$ with $i>0$, then we increase the number of inversions of the sequence, so if we write $ws_i = v$, then $\ell(v) = \ell(w) + 1$ \cite[\S 5.4, Theorem]{humphreys}, so the resulting expression for $w$ will not be minimal. If we choose $i$ maximal so that $w = w' s_{i-1} \cdots s_1 s_0$ with $\ell(w) = \ell(w') + i$, then we have replaced the first column of $\lambda$ with $2n+2-\lambda^\dagger_1$ and then moved it over to the right as much as possible (adding 1 to it each time we pass a column and subtracting 1 from the column we just passed) so that the resulting shape is again a Young diagram. This is the same as removing a border strip of length $2(\ell(\lambda) - n - 1)$ with $i$ columns.
\end{proof}

Finally, there is a third modification rule, defined in \cite[\S 2.4]{koiketerada}. We will not need to know the statement of the rule, but we will cite some results from \cite{koiketerada}, so we need to know that their rule is equivalent to the previous two. The equivalence of the rule from \cite[\S 2.4]{koiketerada} with the border strip rule comes from the fact that both rules were derived from the same determinantal formulas (see \cite[Theorem 1.3.3]{koiketerada} and \cite[Footnote 18]{king}).

\subsection{The main theorem}
\label{ss:sthm}

Our main theorem is the following:

\begin{theorem}
\label{s:mainthm}
For a partition $\lambda$ and an integer $i$ we have
\begin{displaymath}
\rH_i(L^{\lambda}_{\bullet})=\begin{cases}
\bS_{[\tau_{2n}(\lambda)]}(V) & \textrm{if $i=i_{2n}(\lambda)$} \\
0 & \textrm{otherwise.}
\end{cases}
\end{displaymath}
In particular, if $i_{2n}(\lambda)=\infty$ then $L^{\lambda}_{\bullet}$ is exact.
\end{theorem}

\begin{remark}
Consider the coordinate ring $R$ of rank $\le 2n$ skew-symmetric matrices; identifying $E$ with its dual, this is the quotient of $A$ by the ideal generated by $2(n+1) \times 2(n+1)$ Pfaffians.  A description of the resolution of $R$ over $A$ can be found in \cite[\S 6.4]{weyman} and \cite[\S 3]{jpw}.  On the other hand, $R$ is the $\Sp(V)$-invariant part of $B$, and so the above theorem, combined with \eqref{s:tor}, shows that $\bS_\lambda(E)$ appears in its resolution if and only if $\tau_{2n}(\lambda) = \emptyset$. Thus the modification rule gives an alternative description of the resolution of $R$. It is a pleasant combinatorial exercise to show directly that these two descriptions agree. While the description in terms of the modification rule is more complicated, it has the advantage that it readily generalizes to our situation.
\end{remark}

The proof of the theorem will take the remainder of this section.  We follow the three-step plan outlined in \S\ref{ss:proofoverview}.  Throughout, the space $V$ is fixed and $n=\tfrac{1}{2} \dim(V)$.

\Step{a}
Let $Q_{-1}$ be the set of partitions $\lambda$ whose Frobenius coordinates $(a_1, \ldots, a_r \vert b_1, \ldots, b_r)$ satisfy $a_i=b_i-1$ for all $i$.  This set admits an inductive definition that will be useful for us and which we now describe.  The empty partition belongs to $Q_{-1}$.  A non-empty partition $\mu$ belongs to $Q_{-1}$ if and only if the number of rows in $\mu$ is one more than the number of columns, i.e., $\ell(\mu)=\mu_1+1$, and the partition obtained by deleting the first row and column of $\mu$, i.e., $(\mu_2-1, \ldots, \mu_{\ell(\mu)}-1)$, belongs to $Q_{-1}$.  The significance of this set is the plethysm
\begin{displaymath}
\lw^{\bullet}(\lw^2(E))=\bigoplus_{\mu \in Q_{-1}} \bS_{\mu}(E)
\end{displaymath}
(see \cite[I.A.7, Ex.~4]{macdonald}).

Let $\lambda$ be a partition with $\ell(\lambda) \le n$.  We write $(\lambda \vert \mu)$ in place of $(\lambda \mid_n \mu)$ in this section.  Define
\begin{displaymath}
\begin{split}
S_1(\lambda) &= \{ \textrm{$\mu \in Q_{-1}$ such that $(\lambda \vert \mu)$ is regular} \} \\
S_2(\lambda) &= \{ \textrm{partitions $\alpha$ such that $\tau_{2n}(\alpha)=\lambda$} \}.
\end{split}
\end{displaymath}

\begin{lemma}
\label{lem:s1}
Let $\mu$ be a non-zero partition in $S_1(\lambda)$ and let $\nu$ be the partition obtained by removing the first row and column of $\mu$.  Then $\nu$ also belongs to $S_1(\lambda)$.  Furthermore, let $w$ (resp.\ $w'$) be the unique element of $W$ such that $\alpha=w \bullet (\lambda \vert \mu)$ (resp.\ $\beta=w' \bullet (\lambda \vert \nu)$) is a partition.  Then the border strip $R_{\alpha}$ is defined (see \S\ref{s:modrule}) and we have the following identities:
\begin{displaymath}
\vert R_{\alpha} \vert=2 \mu_1, \qquad
\alpha \setminus R_{\alpha}=\beta, \qquad
c(R_{\alpha})=\mu_1+\ell(w')-\ell(w).
\end{displaymath}
\end{lemma}

\begin{proof}
Suppose that in applying Bott's algorithm to $(\lambda \vert \mu)$ the number $\mu_1$ moves $r$ places to the left.  Thus, after the first $r$ steps of the algorithm, we reach the sequence
\begin{displaymath}
(\lambda_1, \ldots, \lambda_{n-r}, \mu_1-r, \lambda_{n-r+1}+1, \ldots, \lambda_n+1, \mu_2,
\ldots, \mu_{\ell(\mu)}).
\end{displaymath}
Notice that the subsequence starting at $\lambda_{n-r+1}+1$ is the same as the subsequence of $(\lambda \vert \nu)$ starting at $\lambda_{n-r+1}$, except 1 has been added to each entry of the former.  It follows that Bott's algorithm runs in exactly the same manner on each.  In particular, if $(\lambda \vert \nu)$ were not regular then $(\lambda \vert \mu)$ would not be either; this shows that $\nu$ belongs to $S_1(\lambda)$.  Suppose that Bott's algorithm on $(\lambda \vert \nu)$ terminates after $N=\ell(w')$ steps.  By the above discussion, Bott's algorithm on $(\lambda \vert \mu)$ terminates after $N+r=\ell(w)$ steps, and we have the following formula for $\alpha$:
\begin{displaymath}
\alpha_i=\begin{cases}
\lambda_i & 1 \le i \le n-r \\
\mu_1-r & i=n-r+1 \\
\beta_{i-1}+1 & n-r+2 \le i \le n+\mu_1+1
\end{cases}
\end{displaymath}
Since $\ell(\alpha)=n+\mu_1+1$, the border strip $R_{\alpha}$ has $2 \mu_1$ boxes.  Using Remark~\ref{rmk:striphook}, we see that $R_{\alpha}$ exists since the box in the $(n-r+1)$th row and the first column has a hook of size $2(\ell(\alpha) - n - 1)$. Furthermore, $\alpha \setminus R_{\alpha}=\beta$ and $c(R_{\alpha})=\mu_1-r$. Since $r=\ell(w)-\ell(w')$, the result follows.
\end{proof}

\begin{lemma}
\label{lem:s2}
Let $\nu$ belong to $S_1(\lambda)$ and suppose that $w' \in W$ is such that $w' \bullet (\lambda \vert \nu)=\beta$ is a partition.  Let $\alpha$ be a partition such that $R_{\alpha}$ is defined and $\alpha \setminus R_{\alpha}=\beta$.  Then there exists a partition $\mu \in S_1(\lambda)$ and an element $w \in W$ such that $w \bullet (\lambda \vert \mu)=\alpha$, and the partition obtained from $\mu$ by removing the first row and column is $\nu$.
\end{lemma}

\begin{proof}
Reverse the steps of Lemma~\ref{lem:s1}.
\end{proof}

\begin{proposition}
\label{s:stepa}
There is a unique bijection $S_1(\lambda) \to S_2(\lambda)$ under which $\mu$ maps to $\alpha$ if there exists $w \in \fS$ such that $w \bullet (\lambda \vert \mu)=\alpha$; in this case, $\ell(w)+i_{2n}(\alpha)=\tfrac{1}{2} \vert \mu \vert$.
\end{proposition}

\begin{proof}
Let $\mu$ be an element of $S_1(\lambda)$ and let $w \in W$ be such that $w \bullet (\lambda \vert \mu)=\alpha$ is a partition.  We show by induction on $\vert \mu \vert$ that $\alpha$ belongs to $S_2(\lambda)$ and that $\ell(w)+i_{2n}(\alpha)=\tfrac{1}{2} \vert \mu \vert$.  For $\vert \mu \vert=0$ this is clear:  $w=1$ and $\alpha=\lambda$.  Suppose now that $\mu$ is non-empty.  In what follows, we tacitly employ Lemma~\ref{lem:s1}.  Let $\nu$ be the partition obtained by removing the first row and column of $\mu$.  Then $\nu$ belongs to $S_1(\lambda)$, and so we can choose $w' \in W$ such that $w' \bullet (\lambda \vert \nu)=\beta$ is a partition.  By induction we have $\tau_{2n}(\beta)=\lambda$ and $\ell(w')+i_{2n}(\beta)=\tfrac{1}{2} \vert \nu \vert$.  Since $\alpha \setminus R_{\alpha}=\beta$, we have $\tau_{2n}(\alpha)=\tau_{2n}(\beta)=\lambda$.  Furthermore, $i_{2n}(\alpha)=c(R_{\alpha})+i_{2n}(\beta)$, and so
\begin{displaymath}
i_{2n}(\alpha)=\mu_1+\ell(w')-\ell(w)+i_{2n}(\beta)=\tfrac{1}{2} \vert \mu \vert - \ell(w).
\end{displaymath}
This completes the induction.

We have thus shown that $\mu \mapsto \alpha$ defines a map of sets $S_1(\lambda) \to S_2(\lambda)$.  We now show that this map is injective. Suppose $\mu$ and $\mu'$ are two elements of $S_1(\lambda)$ that both map to $\alpha$. Then the sequences $(\lambda|\mu) + \rho$ and $(\lambda|\mu') + \rho$ are identical as multisets of numbers. In particular, we can rearrange the sequence of numbers to the right of the bar of $(\lambda | \mu) + \rho$ to get the sequence of numbers to the right of the bar of $(\lambda | \mu') + \rho$. But both of these sequences (to the right of the bar) are strictly decreasing, so we see that $\mu = \mu'$.

Finally, we show that $\mu \mapsto \alpha$ is surjective.  The partition $\alpha=\lambda$ has the empty partition as its preimage.  Suppose now that $\alpha \ne \lambda$ belongs to $S_2(\lambda)$, and let $\beta=\alpha \setminus R_{\alpha}$.  By induction on size, we can find $\nu \in S_1(\lambda)$ mapping to $\beta$.  Applying Lemma~\ref{lem:s2}, we find a partition $\mu \in S_1(\lambda)$ mapping to $\alpha$.  This completes the proof.
\end{proof}

\Step{b}
Let $E$ be a vector space of dimension at least $n$.  Let $X$ be the Grassmannian of rank $n$ quotients of $E$.  Let $\cR$ and $\cQ$ be the tautological bundles on $X$ as in \eqref{eqn:taut}.  Put $\eps=\lw^2(E) \otimes \cO_X$, $\xi=\lw^2{\cR}$ and define $\eta$ by the exact sequence
\begin{displaymath}
0 \to \xi \to \eps \to \eta \to 0.
\end{displaymath}
Finally, for a partition $\lambda$ with $\ell(\lambda) \le n$, put $\cM_{\lambda}=\Sym(\eta) \otimes \bS_{\lambda}(\cQ)$ and $M_{\lambda}=\rH^0(X, \cM_{\lambda})$.  Note that $A=\rH^0(X, \Sym(\eps))$, and so $M_{\lambda}$ is an $A$-module.

\begin{lemma}
Let $\lambda$ be a partition with $\ell(\lambda) \le n$ and $\mu \in S_1(\lambda)$ correspond to $\alpha \in S_2(\lambda)$.  Then
\begin{displaymath}
\rH^i(X, \bS_{\lambda}(\cQ) \otimes \bS_{\mu}(\cR))=\begin{cases}
\bS_{\alpha}(E) & \textrm{if $i=\tfrac{1}{2} \vert \mu \vert-i_{2n}(\alpha)$} \\
0 & \textrm{otherwise.}
\end{cases}
\end{displaymath}
\end{lemma}

\begin{proof}
This follows immediately from Proposition~\ref{s:stepa} and the Borel--Weil--Bott theorem.
\end{proof}

\begin{lemma}
Let $\lambda$ be a partition with $\ell(\lambda) \le n$ and let $i$ be an integer.  We have
\begin{displaymath}
\bigoplus_{j \in \bZ} \rH^j(X, \lw^{i+j}(\xi) \otimes \bS_{\lambda}(\cQ)) = \bigoplus_{\alpha} \bS_{\alpha}(E),
\end{displaymath}
where the sum is over partitions $\alpha$ with $\tau_{2n}(\alpha)=\lambda$ and $i_{2n}(\alpha)=i$.  In particular, when $i<0$ the left side above vanishes.
\end{lemma}

\begin{proof}
We have
\begin{displaymath}
\lw^{i+j}(\xi)=\lw^{i+j}(\lw^2(\cR))=\bigoplus_{\substack{\mu \in Q_{-1},\\ \vert \mu \vert=2(i+j)}} \bS_{\mu}(\cR),
\end{displaymath}
and so
\begin{displaymath}
\bigoplus_{j \in \bZ} \rH^j(X, \lw^{i+j}(\xi) \otimes \bS_{\lambda}(\cQ))
= \bigoplus_{\mu \in Q_{-1}} \rH^{\vert \mu \vert/2-i}(X, \bS_{\mu}(\cR) \otimes \bS_{\lambda}(\cQ)).
\end{displaymath}
The result now follows from the previous lemma.
\end{proof}

\begin{proposition}
\label{s:stepb}
We have
\begin{displaymath}
\Tor^A_i(M_{\lambda}, \bC)=\bigoplus_{\alpha} \bS_{\alpha}(E),
\end{displaymath}
where the sum is over partitions $\alpha$ with $\tau_{2n}(\alpha)=\lambda$ and $i_{2n}(\alpha)=i$.
\end{proposition}

\begin{proof}
This follows immediately from the previous lemma and Proposition~\ref{geo}.
\end{proof}

\Step{c}
For a partition $\lambda$ with at most $n$ parts put $B_{\lambda}=\Hom_{\Sp(V)}(\bS_{[\lambda]}(V), B)$.  Note that $B_{\lambda}$ is an $A$-module and has a compatible action of $\GL(E)$.  Our goal is to show that $B_{\lambda}$ is isomorphic to $M_{\lambda}$.

\begin{lemma} \label{lem:B=M}
The spaces $B_{\lambda}$ and $M_{\lambda}$ are isomorphic as representations of $\GL(E)$ and have finite multiplicities.
\end{lemma}

\begin{proof}
Let $M=\bigoplus_{\ell(\lambda) \le n} M_{\lambda} \otimes \bS_{[\lambda]}(V)$.  It is enough to show that $M$ and $B$ are isomorphic as representations of $\GL(E) \times \Sp(V)$ and have finite multiplicities.  In fact, it is enough to show that the $\bS_{\theta}(E)$ multiplicity spaces of $M$ and $B$ are isomorphic as representations of $\Sp(V)$ and have finite multiplicities.  This is what we do.

The $\bS_{\theta}(E)$ multiplicity space of $B$ is $\bS_{\theta}(V)$.  The decomposition of this in the representation ring of $\Sp(V)$ can be computed by applying the specialization homomorphism to \cite[Thm.~2.3.1(1)]{koiketerada}.  The result is
\begin{displaymath}
\sum_{\mu, \nu} (-1)^{i_{2n}(\nu)} c^{\theta}_{(2\mu)^{\dag},\nu} [\bS_{[\tau_{2n}(\nu)]}(V)].
\end{displaymath}
Note that for a fixed $\theta$ there are only finitely many values for $\mu$ and $\nu$ which make the Littlewood--Richardson coefficient non-zero, which establishes finiteness of the multiplicities.  Now, we have an equality
\begin{displaymath}
[M_{\lambda}]=[A] \sum_{i \ge 0} (-1)^i [\Tor^A_i(M_{\lambda}, \bC)]
\end{displaymath}
in the representation ring of $\GL(E)$.  Applying Proposition~\ref{s:stepb}, we find
\begin{displaymath}
\sum_{i \ge 0} (-1)^i [\Tor^A_i(M_{\lambda}, \bC)] = \sum_{\tau_{2n}(\nu)=\lambda} (-1)^{i_{2n}(\nu)} [\bS_{\nu}(E)].
\end{displaymath}
As $[A]=\sum_{\mu} [\bS_{(2\mu)^{\dag}}(E)]$ (see \cite[I.A.7, Ex.~2]{macdonald}), we obtain
\begin{displaymath}
[M_{\lambda}]=\sum_{\substack{\nu, \mu, \theta\\ \tau_{2n}(\nu)=\lambda}}(-1)^{i_{2n}(\nu)} c^{\theta}_{(2\mu)^{\dag},\nu} [\bS_{\theta}(E)].
\end{displaymath}
We therefore find that the $\bS_{\theta}(E)$-component of $M$ is given by
\begin{displaymath}
\sum_{\substack{\lambda,\mu\\ \tau_{2n}(\nu)=\lambda}} (-1)^{i_{2n}(\nu)} c^{\theta}_{(2\mu)^{\dag},\nu} [\bS_{[\lambda]}(V)].
\end{displaymath}
The result now follows.
\end{proof}

\begin{proposition}
\label{s:stepc}
We have an isomorphism $M_{\lambda} \to B_{\lambda}$ which is $A$-linear and $\GL(E)$-equivariant.
\end{proposition}

\begin{proof}
According to Proposition~\ref{s:stepb}, we have 
\begin{align*}
\Tor^A_0(M_{\lambda}, \bC) &=\bS_{\lambda}(E),\\
\Tor^A_1(M_\lambda, \bC) &=\bS_{(\lambda,1^{2n+2-2\ell(\lambda)})}(E),
\end{align*}
since the only $\nu$ for which $\tau_{2n}(\nu) = \lambda$ and $i_{2n}(\nu) \le 1$ must agree with $\lambda$ everywhere except possibly the first column. We therefore have a presentation
\begin{displaymath}
\bS_{(\lambda,1^{2n+2-2\ell(\lambda)})}(E) \otimes A \to \bS_{\lambda}(E) \otimes A \to M_{\lambda} \to 0.
\end{displaymath}
Note that $\bS_{(\lambda, 1^{2n+2-2\ell(\lambda)})}(E)$ occurs with multiplicity one in $\bS_{\lambda}(E) \otimes A$, and thus does not occur in $M_{\lambda}$; it therefore does not occur in $B_{\lambda}$ either, since $M_{\lambda}$ and $B_{\lambda}$ are isomorphic as representations of $\GL(E)$ by Lemma~\ref{lem:B=M}.

Now, $\Tor^A_0(B, \bC)$ is the coordinate ring of the Littlewood variety, and its $\bS_{[\lambda]}(V)$ multiplicity space is $\bS_{\lambda}(E)$ by \eqref{s:cring}.  We therefore have a surjection $f \colon \bS_{\lambda}(E) \otimes A \to B_{\lambda}$.  Since $\bS_{(\lambda,1^{2n+2-2\ell(\lambda)})}(E)$ does not occur in $B_{\lambda}$, the copy of $\bS_{(\lambda,1^{2n+2-2\ell(\lambda)})}(E)$ in $\bS_{\lambda}(E) \otimes A$ lies in the kernel of $f$, and therefore $f$ induces a surjection $M_{\lambda} \to B_{\lambda}$.  Finally, since the two are isomorphic as $\GL(E)$ representations and have finite multiplicity spaces, this surjection is an isomorphism.
\end{proof}

Combining this proposition with Proposition~\ref{s:stepb}, we obtain the following corollary.

\begin{corollary}
We have
\begin{displaymath}
\Tor^A_i(B, \bC)=\bigoplus_{i_{2n}(\lambda)=i} \bS_{\lambda}(E) \otimes \bS_{[\tau_{2n}(\lambda)]}(V).
\end{displaymath}
\end{corollary}

Combining this with \eqref{s:tor} yields the main theorem.  (We can choose $E$ to be arbitrarily large.)

\begin{remark}
The arguments of step~c made no use of the construction of the module $M_{\lambda}$, simply that it satisfied Proposition~\ref{s:stepb}.  More precisely, say that a $\GL(E)$-equivariant $A$-module $M$ is of ``type $\lambda$'' (for an admissible partition $\lambda$) if
\begin{displaymath}
\Tor^A_i(M, \bC)=\bigoplus_{\alpha} \bS_{\alpha}(E),
\end{displaymath}
where the sum is over all partitions $\alpha$ with $\tau_{2n}(\alpha)=\lambda$ and $i_{2n}(\alpha)=i$.  Then the arguments of step~c establish the following statement:  if a type $\lambda$ module exists then it is isomorphic to $B_{\lambda}$, and thus $B_{\lambda}$ has type $\lambda$.  (Actually the argument is a bit weaker, since it works with all $\lambda$ at once.)  Step~b can be thought of as simply providing a construction of a module of type $\lambda$.\end{remark}

\subsection{Examples}
\label{ss:examples}

We now give a few examples to illustrate the theorem.  

\begin{example}
Suppose $\lambda=(1^i)$.  Then $L^{\lambda}_\bullet$ is the complex $\lw^{i-2}{V} \to \lw^i{V}$, where the differential is the multiplication by the symplectic form on $V^*$ treated as an element of $\lw^2V$. 
\begin{itemize}
\item If $i\le n$ then the differential is injective, and $\rH_0(L^{\lambda}_\bullet)=\bS_{[1^i]}(V)$ is an irreducible representation of $V$. 
\item If $i=n+1$ then the differential is an isomorphism, and all homology of $L^{\lambda}_\bullet$ vanishes. 
\item If $n+2\le i\le 2n+2$ then the differential is surjective and $\rH_1(L^{\lambda}_\bullet)=\bS_{[1^{2n-i+2}]}(V)$. 
\item If $i>2n+2$ then the complex $L^{\lambda}_\bullet$ is identically 0. \qedhere
\end{itemize}
\end{example}

\begin{example}
Suppose $\lambda=(2,1,1)$.  Then $L^{\lambda}_\bullet$ is the complex $\bC \to \bS_{(2,1,1)/(1,1)}(V) \to \bS_{(2,1,1)}(V)$, where the differential is the multiplication by the symplectic form on $V^*$ treated as an element of $\lw^2V$.  
\begin{itemize}
\item If $n\ge 3$ then the differential is injective, and $\rH_0(L^{\lambda}_\bullet)=\bS_{[2,1,1]}(V)$ is an irreducible representation of $V$.  
\item If $n=2$ then the complex is exact, and all homology of $L^{\lambda}_\bullet$ vanishes.  
\item If $n=1$ then $\rH_1(L^{\lambda}_\bullet)=\bS_{[2]}(V)$. 
\item Finally when $n=0$ then $\rH_2(L^{\lambda}_\bullet)=\bC$. \qedhere
\end{itemize}
\end{example}

The reader will check easily that in both instances the description of the homology agrees with the rule given by the Weyl group action.

\begin{example}
\label{ex:s3}
Suppose $\lambda=(6,5,4,4,3,3,2)$ and $n=2$ (so $\dim(V)=4$).  The modification rule, using border strips, proceeds as follows:
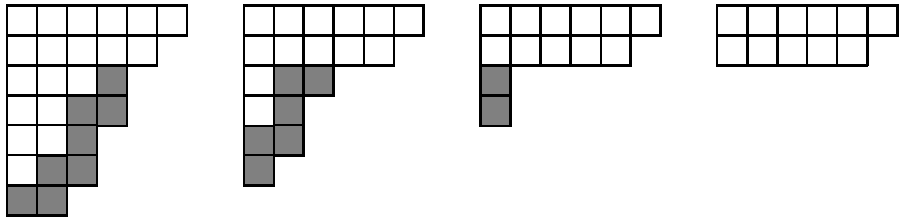
\begin{displaymath}
\ydiagram[*(white)]{6,5,3,2,2,1}*[*(gray)]{6,5,4,4,3,3,2} \qquad
\ydiagram[*(white)]{6,5,1,1}*[*(gray)]{6,5,3,2,2,1} \qquad
\ydiagram[*(white)]{6,5}*[*(gray)]{6,5,1,1} \qquad
\ydiagram{6,5}
\end{displaymath}
We start on the left with $\lambda_0=\lambda$.  As $\ell(\lambda_0)=7$, we are supposed to remove the border strip $R_0$ of size $2(\ell(\lambda_0)-n-1)=8$; this border strip is shaded.  The result is the second displayed partition, $\lambda_1=(6,5,3,2,2,1)$.  As $\ell(\lambda_1)=6$, the border strip $R_1$ we remove from it has length 6.  The result of removing this strip is the third partition $\lambda_2=(6,5,1,1)$.  As $\ell(\lambda_2)=4$, the border strip $R_2$ has length 2.  The result of removing it is the final partition $\lambda_3=(6,5)$.  This satisfies $\ell(\lambda_3) \le n$, so the algorithm stops.  We thus see that $\tau_4(\lambda)=(6,5)$ and
\begin{displaymath}
i_4(\lambda)=c(R_0)+c(R_1)+c(R_2)=4+3+1=8.
\end{displaymath}
It follows that $\rH_i(L_\bullet^{\lambda})=0$ for $i \ne 8$ and $\rH_8(L_\bullet^{\lambda})=\bS_{[6,5]}(\bC^4)$.

Now we illustrate the modification rule using the Weyl group action. We write $\alpha \xrightarrow{s_i} \beta$ if $\beta = s_i(\alpha)$. The idea for getting the Weyl group element is to apply $s_0$ if the first column length is too long, then sort the result, and repeat as necessary. We start with $\lambda^\dagger + \rho = (7,7,6,4,2,1) + (-3,-4,-5,\dots)$:
\begin{align*}
(4,3,1,-2,-5,-7) &\xrightarrow{s_0} (-4,3,1,-2,-5,-7)
\xrightarrow{s_1} (3,-4,1,-2,-5,-7)\\
&\xrightarrow{s_2} (3,1,-4,-2,-5,-7)
\xrightarrow{s_3} (3,1,-2,-4,-5,-7)\\
&\xrightarrow{s_0} (-3,1,-2,-4,-5,-7)
\xrightarrow{s_1} (1,-3,-2,-4,-5,-7)\\
&\xrightarrow{s_2} (1,-2,-3,-4,-5,-7)
\xrightarrow{s_0} (-1,-2,-3,-4,-5,-7).
\end{align*}
Subtracting $\rho$ from the result, we get $(6,5)^\dagger$.
\end{example}

\section{Orthogonal groups}

\subsection{Representations of $\bO(V)$}

Let $(V, \omega)$ be an orthogonal space of dimension $m$ (here $\omega \in \Sym^2 V^*$ is the orthogonal form, and gives an isomorphism $V \cong V^*$). We write $m = 2n$ if it is even, or $m = 2n+1$ if it is odd.  We now recall the representation theory of $\bO(V)$; see \cite[\S 19.5]{fultonharris} for details.  The irreducible representations of $\bO(V)$ are indexed by partitions $\lambda$ such that the first two columns have at most $m$ boxes in total, i.e., $\lambda^{\dag}_1+\lambda^{\dag}_2 \le m$.  We call such partitions {\bf admissible}.  For an admissible partition $\lambda$, we write $\bS_{[\lambda]}(V)$ for the corresponding irreducible representation of $\bO(V)$.

Given an admissible partition $\lambda$, we let $\lambda^{\sigma}$ be the partition obtained by changing the number of boxes in the first column of $\lambda$ to $m$ minus its present value; that is, $(\lambda^{\sigma})^{\dag}_1=m-\lambda^{\dag}_1$.  We call $\lambda^{\sigma}$ the {\bf conjugate} of $\lambda$.  Conjugation defines an involution on the set of admissible partitions.  On irreducible representations, conjugating the partition corresponds to twisting by the sign character: $\bS_{[\lambda^{\sigma}]}(V)=\bS_{[\lambda]}(V) \otimes \sgn$.  It follows that $\bS_{[\lambda]}(V)$ and $\bS_{[\lambda^{\sigma}]}(V)$ are isomorphic when restricted to $\SO(V)$.  In fact, these restrictions remain irreducible, unless $\lambda=\lambda^{\sigma}$ (which is equivalent to $\ell(\lambda)=n$ and $m=2n$), in which case $\bS_{[\lambda]}(V)$ decomposes as a sum of two non-isomorphic irreducible representations.

For an admissible partition $\lambda$, exactly one element of the set $\{\lambda, \lambda^{\sigma}\}$ has at most $n$ boxes in its first column.  We denote this element by $\ol{\lambda}$.  Thus $\ol{\lambda}=\lambda$ if $\lambda_1^{\dag} \le n$, and $\ol{\lambda}=\lambda^{\sigma}$ otherwise.

\subsection{The Littlewood complex} \label{ss:bdcx}

Let $E$ be a vector space.  Put $U=\Sym^2(E)$, $A=\Sym(U)$ and $B=\Sym(E \otimes V)$.  Consider the inclusion $U \subset B$ given by
\begin{displaymath}
\Sym^2(E) \subset \Sym^2(E) \otimes \Sym^2(V) \subset \Sym^2(E \otimes V),
\end{displaymath}
where the first inclusion is multiplication with $\omega$.  This inclusion defines an algebra homomorphism $A \to B$.  Put $C=B \otimes_A \bC$; this is the quotient of $B$ by the ideal generated by $U$.  We have maps
\begin{displaymath}
\Spec(C) \to \Spec(B) \to \Spec(A).
\end{displaymath}
We have a natural identification of $\Spec(B)$ with the space $\Hom(E, V)$ of linear map $\varphi \colon E \to V$ and of $\Spec(A)$ with the space $\Sym^2(E)^*$ of symmetric forms on $E$.  The map $\Spec(B) \to \Spec(A)$ takes a linear map $\varphi$ to the form $\varphi^*(\omega)$.  The space $\Spec(C)$, which we call that {\bf Littlewood variety}, is the scheme-theoretic fiber of this map above 0, i.e., is consists of those maps $\varphi$ such that $\varphi^*(\omega)=0$. In other words, $\Spec(C)$ consists of maps $\phi \colon E \to V$ such that the image of $\phi$ is an isotropic subspace of $V$.

Let $\rK_{\bullet}=B \otimes \lw^{\bullet}{U}$ be the Koszul complex of the Littlewood variety.  We can decompose this complex under the action of $\GL(E)$:
\begin{displaymath}
\rK_{\bullet}(E)=\bigoplus_{\ell(\lambda) \le \dim{E}} \bS_{\lambda}(E) \otimes L^{\lambda}_{\bullet}.
\end{displaymath}
The complex $L^{\lambda}_{\bullet}$ is the {\bf Littlewood complex}, and is independent of $E$ (so long as $\dim{E} \ge \ell(\lambda)$). By \cite[Proposition 3.6.3]{howe}, its zeroth homology is
\begin{equation}
\label{o:Hzero}
\rH_0(L^{\lambda}_{\bullet})=\begin{cases}
\bS_{[\lambda]}(V) & \textrm{if $\lambda$ is admissible} \\
0 & \textrm{otherwise.}
\end{cases}
\end{equation}
By Lemma~\ref{lem:koszul}, we have $\rH_i(\rK_{\bullet})=\Tor^A_i(B, \bC)$, and so we have a decomposition
\begin{equation}
\label{o:tor}
\Tor^A_i(B, \bC)=\bigoplus_{\ell(\lambda) \le \dim{E}} \bS_{\lambda}(E) \otimes \rH_i(L^{\lambda}_{\bullet}).
\end{equation}
Applied to $i=0$, we obtain
\begin{equation}
\label{o:cring}
C=\bigoplus_{\textrm{admissible $\lambda$}} \bS_{\lambda}(E) \otimes \bS_{[\lambda]}(V).
\end{equation}

\subsection{A special case of the main theorem}

Our main theorem computes the homology of the complex $L^{\lambda}_{\bullet}$.  We now formulate and prove the theorem in a particularly simple case.  We mention this here only because it is worthwhile to know; the argument is not needed to prove the main theorem.

\begin{proposition}
Suppose $\lambda$ is admissible.  Then
\begin{displaymath}
\rH_i(L^{\lambda}_{\bullet}) = \begin{cases}
\bS_{[\lambda]}(V) & \textrm{if $i=0$} \\
0 & \textrm{otherwise.}
\end{cases}
\end{displaymath}
\end{proposition}

\begin{proof}
Choose $E$ to be of dimension $n$.  By Lemma~\ref{lem:BD:CI} below, $\rK_{\bullet}(E)$ has no higher homology.  It follows that $L^{\lambda}_{\bullet}$ does not either.  The computation of $\rH_0(L^{\lambda}_{\bullet})$ is given in \eqref{o:Hzero}.
\end{proof}

\begin{lemma} \label{lem:BD:CI}
Suppose $\dim{E} \le n$.  Then $U \subset B$ is spanned by a regular sequence.
\end{lemma}

\begin{proof}
The proof is the same as Lemma~\ref{lem:C:CI}. The only difference worth pointing out (but which does not affect the proof) is that when $\dim V = 2n$ and $\dim E = n$, the Grassmannian of isotropic $n$-dimensional subspaces of $V$ has two connected components, and the variety cut out by $U$ has two irreducible components.
\end{proof}

\subsection{The modification rule}
\label{o:modrule}

As in the symplectic case, we now associate to a partition $\lambda$ two quantities $i_m(\lambda)$ and $\tau_m(\lambda)$.  We again give two equivalent definitions.

We begin with the Weyl group definition, following \cite[\S 1.4]{wenzl}.  Let $s_0$ be the automorphism of the set $\cU$ which negates and swaps the first and second entries, and let $W$ be the group generated by the $s_i$ with $i \ge 0$. This is a Coxeter group of type $\rD_\infty$. Let $\ell \colon W \to \bZ_{\ge 0}$ be the length function, which is defined just as in \eqref{eqn:word}.  Note that this group $W$, as a subgroup of $\Aut(\cU)$, is equal to the one from \S\ref{s:modrule}, but that the length function is different since we are using a different set of simple reflections.  Let $\rho=(-m/2, -m/2-1, \ldots)$.  Define a new action of $W$ on $\cU$ by $w \bullet \lambda=w(\lambda+\rho)-\rho$.  On $\fS$ this agrees with the one defined in \S\ref{ss:bott}.  The action of $s_0$ is given by
\begin{displaymath}
s_0 \bullet (a_1, a_2, a_3, \ldots) = (m+1-a_2, m+1-a_1, a_3, \ldots).
\end{displaymath}
The definitions of $i_m(\lambda)$ and $\tau_m(\lambda)$ are now exactly as in the first half of \S \ref{s:modrule}.

~

We now give the border strip definition. This is motivated by \cite[\S 5]{sundaram} (which is based on \cite{king}), but \cite{sundaram} only focuses on the special orthogonal group, so we have to modify the definition to get the correct answer for the full orthogonal group. This is the same as the one given in \S \ref{s:modrule}, except for three differences:
\begin{enumerate}[(D1)]
\item the border strip $R_{\lambda}$ has length $2\ell(\lambda)-m$,
\item in the definition of $i_m(\lambda)$, we use $c(R_{\lambda})-1$ instead of $c(R_{\lambda})$, and
\item if the total number of border strips removed is odd, then replace the end result $\mu$ with $\mu^\sigma$.
\end{enumerate}
One can stop applying the modification rule either when $\lambda$ becomes admissible or when $\ell(\lambda) \le n$; the resulting values of $\tau$ and $i$ are the same.  For instance, if $\lambda$ is admissible but $\ell(\lambda)>n$ then one can stop immediately with $i=0$ and $\tau=\lambda$.  Instead, one could remove a border strip.  This border strip occupies only the first column and when removed yields $\lambda^{\sigma}$.  Thus $i=0$ and by (D3), since we removed an odd number of strips, $\tau=(\lambda^{\sigma})^{\sigma}=\lambda$.

\begin{proposition}
The above two definitions agree.
\end{proposition}

\begin{proof}
Suppose that we are removing a border strip $R_1$ of size $2\ell(\lambda) - m$ from $\lambda$ which begins at the first box in the final row of $\lambda$. Let $c_1 = c(R_1)$ be the number of columns of $R_1$. The first two column lengths of $\lambda \setminus R_1$ are $(\lambda^\dagger_2 - 1, \lambda^\dagger_3 - 1)$. We have two cases depending on which of the two quantities $\lambda^\dagger_2 + \lambda^\dagger_3 - 2$ and $m$ is bigger.

First suppose that $\lambda^\dagger_2 + \lambda^\dagger_3 - 2 > m$. Then we remove another border strip $R_2$ of size $2(\lambda^\dagger_2 - 1) - m$ from $\lambda \setminus R_1$ which begins at the first box in the final row. Let $c_2 = c(R_2)$ be the number of columns of $R_2$. The sequence 
\[
(s_{c_2-1} s_{c_2-2} \cdots s_2 s_1 s_{c_1-1} s_{c_1-2} \cdots s_3 s_2 s_0) \bullet \lambda^\dagger
\]
gives the column lengths of $(\lambda \setminus R_1) \setminus R_2$.

Now suppose that $\lambda^\dagger_2 + \lambda^\dagger_3 - 2 \le m$. Then we have only removed 1 border strip, which is an odd number, so we have to replace $\lambda \setminus R_1$ with $(\lambda \setminus R_1)^\sigma$ according to (D3) above. In this case, the sequence
\[
(s_{c_1 - 1} s_{c_1 - 2} \cdots s_3 s_2 s_0) \bullet \lambda^\dagger
\]
gives the column lengths of $(\lambda \setminus R_1)^\sigma$.

Conversely, if we use the Weyl group modification rule with $w \in W$, then the expression \eqref{eqn:word} for $w$ must begin with $s_0$: if we apply any $s_i$ with $i>0$, then we increase the number of inversions of the sequence, so if we write $ws_i = v$, then $\ell(v) = \ell(w) + 1$ \cite[\S 5.4, Theorem]{humphreys}, so the resulting expression for $w$ will not be minimal. If we choose $i$ maximal so that $w = w' s_{i-1} \cdots s_3 s_2 s_0$ with $\ell(w) = \ell(w') + i-1$, then we have replaced the first two columns of $\lambda$ with $(m+1-\lambda_2^\dagger, m+1-\lambda_1^\dagger)$ and then moved the column of length $m+1-\lambda_1^\dagger$ over to the right as much as possible (adding 1 to it each time we pass a column and subtracting 1 from the column we just passed) so that the resulting shape (minus the first column) is again a Young diagram. 

Now there are two possibilities: if the whole shape is a Young diagram, then it is the result of first removing a border strip of length $2\ell(\lambda) - m$ with $i$ columns, and then replacing the resulting $\mu$ with $\mu^c$. Otherwise, the first column length of the resulting shape is less than the second column length. If we choose $j$ maximal so that $w' = w'' s_{j-1} \cdots s_2 s_1$ with $\ell(w') = \ell(w'') + j -1$, then we have moved the first column over to the right as much as possible (adding 1 to it each time we pass a column and subtracting 1 from the column we just passed) so that the resulting shape is again a Young diagram. In this case, then we have removed two border strips of size $2\ell(\lambda) - m$ and $2(\lambda^\dagger-1)-m$ with $i$ and $j$ columns, respectively.
\end{proof}

Finally, there is a third modification rule, defined in \cite[\S 2.4]{koiketerada}. As in the symplectic case, we will not need to know the statement of the rule, but we will cite some results from \cite{koiketerada}. The equivalence of the rule from \cite[\S 2.4]{koiketerada} with the border strip rule comes from the fact that both rules were derived from the same determinantal formulas (see \cite[Theorem 1.3.2]{koiketerada} and \cite[Footnote 17]{king}). We remark that both rules are only stated for the special orthogonal group, but this will be enough for our purposes. 

As a matter of notation, we write $\ol{\tau}_m(\lambda)$ in place of $\ol{\tau_m(\lambda)}$.

\subsection{The main theorem}

Our main theorem is exactly the same as in the symplectic case:

\begin{theorem}
\label{o:mainthm}
For a partition $\lambda$ and an integer $i$ we have
\begin{displaymath}
\rH_i(L^{\lambda}_{\bullet})=\begin{cases}
\bS_{[\tau_m(\lambda)]}(V) & \textrm{if $i=i_m(\lambda)$} \\
0 & \textrm{otherwise.}
\end{cases}
\end{displaymath}
In particular, if $i_m(\lambda)=\infty$ then $L^{\lambda}_{\bullet}$ is exact.
\end{theorem}

\begin{remark}
Consider the coordinate ring $R$ of rank $\le m$ symmetric matrices; identifying $E$ with its dual, this is the quotient of $A$ by the ideal generated by $(m+1) \times (m+1)$ minors.  The resolution of $R$ over $A$ is known, see \cite[\S 6.3]{weyman} or \cite[\S 3]{jpw}.  On the other hand, $R$ is the $\bO(V)$-invariant part of $B$, and so the above theorem, combined with \eqref{o:tor}, shows that $\bS_{\lambda}(E)$ appears in its resolution if and only if $\tau_m(\lambda)=\emptyset$. Thus the modification rule gives an alternative description of the resolution of $R$. It is a pleasant combinatorial exercise to show directly that these two descriptions agree. As in the symplectic case, the description in terms of the modification rule is more complicated, but has the advantage of generalizing to our situation.
\end{remark}

We separate the proof of the theorem into two cases, according to whether $m$ is even or odd.  In each case, we follow the three-step plan from \S\ref{ss:proofoverview}.

\subsection{The even case}

Throughout this section, $m=2n$ is the dimension of the space $V$.

\Step{a}
Let $Q_1$ be the set of partitions $\lambda$ whose Frobenius coordinates $(a_1, \ldots, a_r \vert b_1, \ldots, b_r)$ satisfy $a_i=b_i+1$ for each $i$.  This set admits an inductive definition, as follows.  The empty partition belongs to $Q_1$.  A non-empty partition $\mu$ belongs to $Q_1$ if and only if the number of columns in $\mu$ is one more than the number of rows, i.e., $\ell(\mu)=\mu_1-1$, and the partition obtained by deleting the first row and column of $\mu$, i.e., $(\mu_2-1, \ldots, \mu_{\ell(\mu)}-1)$, belongs to $Q_1$.  The significance of this set is the plethysm
\begin{displaymath}
\lw^{\bullet}(\Sym^2(E))=\bigoplus_{\mu \in Q_1} \bS_{\mu}(E)
\end{displaymath}
(see \cite[I.A.7, Ex.~5]{macdonald}).

Let $\lambda$ be a partition with $\ell(\lambda) \le n$.  We write $(\lambda \vert \mu)$ in place of $(\lambda \mid_n \mu)$.  Define
\begin{displaymath}
\begin{split}
\ol{S}_1(\lambda) &= \{ \textrm{$\mu \in Q_1$ such that $(\lambda \vert \mu)$ is regular} \} \\
\ol{S}_2(\lambda) &= \{ \textrm{partitions $\alpha$ such that $\ol{\tau}_m(\alpha)=\lambda$} \}.
\end{split}
\end{displaymath}

\begin{lemma} \label{lem:Dbott}
Let $\mu$ be a non-zero partition in $\ol{S}_1(\lambda)$ and let $\nu$ be the partition obtained by removing the first row and column of $\mu$.  Then $\nu$ also belongs to $\ol{S}_1(\lambda)$.  Furthermore, let $w$ (resp.\ $w'$) be the element of $W$ such that $\alpha=w \bullet (\lambda \vert \mu)$ (resp.\ $\beta=w' \bullet (\lambda \vert \nu)$) is a partition.  Then $R_{\alpha}$ is defined and we have the following identities:
\begin{displaymath}
\vert R_{\alpha} \vert=2\mu_1-2, \qquad
\alpha \setminus R_{\alpha}=\beta, \qquad
c(R_{\alpha})=\mu_1+\ell(w')-\ell(w).
\end{displaymath}
\end{lemma}

\begin{proof}
Suppose that in applying Bott's algorithm to $(\lambda \vert \mu)$ the number $\mu_1$ moves $r$ places to the left.  Thus, after the first $r$ steps of the algorithm, we reach the sequence
\begin{displaymath}
(\lambda_1, \ldots, \lambda_{n-r}, \mu_1-r, \lambda_{n-r+1}, \ldots, \lambda_n+1, \mu_2, \ldots, \mu_{\ell(\mu)}).
\end{displaymath}
As before, Bott's algorithm on this sequence runs just like the algorithm on $(\lambda \vert \nu)$, and so $(\lambda \vert \nu)$ is regular and $\nu$ belongs to $\ol{S}_1(\lambda)$.  Suppose the algorithm on $(\lambda \vert \nu)$ terminates after $N=\ell(w')$ steps.  Then the algorithm on $(\lambda \vert \mu)$ terminates after $N+r=\ell(w)$ steps, and we have the following formula for $\alpha$:
\begin{displaymath}
\alpha_i=\begin{cases}
\lambda_i & 1 \le i \le n-r \\
\mu_1-r & i=n-r+1 \\
\beta_{i-1}+1 & n-r+2 \le i \le n+\mu_1-1
\end{cases}
\end{displaymath}
Since $\ell(\alpha)=n+\mu_1-1$, the border strip $R_{\alpha}$ has $2\mu_1-2$ boxes.  Using Remark~\ref{rmk:striphook}, we see that $R_\alpha$ exists since the box in the $(n-r+1)$th row and the first column has a hook of size $2\ell(\alpha) - m$. Furthermore, $\alpha \setminus R_{\alpha}=\beta$ and $c(R_{\alpha})=\mu_1-r$. Since $r = \ell(w) - \ell(w')$, we are done.
\end{proof}

\begin{proposition} \label{prop:o1:stepa}
There is a unique bijection $\ol{S}_1(\lambda) \to \ol{S}_2(\lambda)$ under which $\mu$ maps to $\alpha$ if there exists $w \in \fS$ such that $w \bullet (\lambda \vert \mu)=\alpha$; in this case, $\ell(w)+i_m(\alpha)=\tfrac{1}{2} \vert \mu \vert$ and 
\[
\tau_m(\alpha) = \begin{cases} \lambda & \text{if $\rank(\mu)$ is even}\\ 
\lambda^\sigma & \text{if $\rank(\mu)$ is odd.}\end{cases}
\]
\end{proposition}

\begin{proof}
Except for the computation of $\tau_m(\alpha)$, the proof is exactly like that of Proposition~\ref{s:stepa}. In the proof of Lemma~\ref{lem:Dbott}, we see that the number of border strips removed from $\alpha$ is $\rank(\mu)$, so the determination of $\tau_m(\alpha)$ follows from (D3) in \S\ref{o:modrule}.
\end{proof}

\Step{b}
Let $E$ be a vector space of dimension at least $n$.  Let $X$ be the Grassmannian of rank $n$ quotients of $E$.  Let $\cR$ and $\cQ$ be the tautological bundles on $X$ as in \eqref{eqn:taut}.  Put $\eps=\Sym^2(E) \otimes \cO_X$, $\xi=\Sym^2(\cR)$ and define $\eta$ by the exact sequence
\begin{displaymath}
0 \to \xi \to \eps \to \eta \to 0.
\end{displaymath}
Finally, for a partition $\lambda$ with $\ell(\lambda) \le n$, put $\ol{\cM}_{\lambda}=\Sym(\eta) \otimes \bS_{\lambda}(\cQ)$ and $\ol{M}_{\lambda} = \rH^0(X, \ol{\cM}_{\lambda})$.  Note that $A=\rH^0(X, \Sym(\eps))$, and so $\ol{M}_{\lambda}$ is an $A$-module.

\begin{proposition}
\label{o:stepb}
We have
\begin{displaymath}
\Tor^A_i(\ol{M}_{\lambda}, \bC)=\bigoplus_{\alpha} \bS_{\alpha}(E),
\end{displaymath}
where the sum is over partitions $\alpha$ with $\ol{\tau}_m(\alpha)=\lambda$ and $i_m(\alpha)=i$.
\end{proposition}

\begin{proof}
The proof is exactly like that of Proposition~\ref{s:stepb}.
\end{proof}

\Step{c}
For an admissible partition $\lambda$ put $B_{\lambda}=\Hom_{\bO(V)}(\bS_{[\lambda]}(V), B)$.  Note that $B_{\lambda}$ is an $A$-module and has a compatible action of $\GL(E)$.  Let $\lambda$ be a partition with at most $n$ rows.  If $\ell(\lambda)=n$, put $\ol{B}_{\lambda}=B_{\lambda}$; otherwise, put $\ol{B}_{\lambda}=B_{\lambda} \oplus B_{\lambda^{\sigma}}$.  Note that the decomposition
\begin{displaymath}
B=\bigoplus_{\ell(\lambda) \le n} \ol{B}_{\lambda} \otimes \bS_{[\lambda]}(V)
\end{displaymath}
holds $\SO(V)$-equivariantly.

\begin{lemma}
\label{o:iso}
Let $\lambda$ be a partition with $\ell(\lambda) \le n$.  The spaces $\ol{B}_{\lambda}$ and $\ol{M}_{\lambda}$ are isomorphic as representations of $\GL(E)$ and have finite multiplicities.
\end{lemma}

\begin{proof}
Put $M=\bigoplus_{\ell(\lambda) \le n} \bS_{[\lambda]}(V) \otimes \ol{M}_{\lambda}$.  It is enough to show that $M$ and $B$ are isomorphic as representations of $\GL(E) \times \SO(V)$ and have finite multiplicities.  In fact, it is enough to show that the $\bS_{\theta}(E)$ multiplicity spaces of $M$ and $B$ are isomorphic as representations of $\SO(V)$ and have finite multiplicities.  This is what we do.

The $\bS_{\theta}(E)$ multiplicity space of $B$ is $\bS_{\theta}(V)$.  The decomposition of this in the representation ring of $\SO(V)$ can be computed by applying the specialization homomorphism to \cite[Thm.~2.3.1(2)]{koiketerada}.  The result is
\begin{displaymath}
\sum_{\mu, \nu} (-1)^{i_m(\nu)} c^{\theta}_{2\mu, \nu} [\bS_{[\ol{\tau}_m(\nu)]}(V)].
\end{displaymath}
For a fixed $\theta$ there are only finitely many values for $\mu$ and $\nu$ which make the Littlewood--Richardson coefficient non-zero, which establishes finiteness of the multiplicities.  Now, we have an equality
\begin{displaymath}
[\ol{M}_{\lambda}]=[A] \sum_{i \ge 0} (-1)^i [\Tor^A_i(\ol{M}_{\lambda}, \bC)]
\end{displaymath}
in the representation ring of $\GL(E)$.  Applying Proposition~\ref{o:stepb}, we find
\begin{displaymath}
\sum_{i \ge 0} (-1)^i [\Tor^A_i(\ol{M}_{\lambda}, \bC)]=\sum_{\ol{\tau}_m(\nu)=\lambda} (-1)^{i_m(\nu)} [\bS_{\nu}(E)].
\end{displaymath}
As $[A]=\sum_{\mu} [\bS_{2\mu}(E)]$ (see \cite[I.A.7, Ex.~1]{macdonald}), we obtain
\begin{displaymath}
[\ol{M}_{\lambda}]=\sum_{\substack{\mu,\nu\\ \ol{\tau}_m(\nu)=\lambda}} (-1)^{i_m(\nu)} c^{\theta}_{2\mu, \nu} [\bS_{\theta}(E)].
\end{displaymath}
We therefore find that the $\bS_{\theta}(E)$ multiplicity space of $M$ is given by
\begin{displaymath}
\sum_{\substack{\lambda,\mu,\nu\\ \tau_m(\nu)=\lambda}} (-1)^{i_m(\nu)} c^{\theta}_{2\mu, \nu} [\bS_{[\lambda]}(V)].
\end{displaymath}
The result now follows.
\end{proof}

\begin{proposition}
\label{o:stepc}
Let $\lambda$ be a partition with $\ell(\lambda) \le n$.  We have an isomorphism $\ol{M}_{\lambda} \to \ol{B}_{\lambda}$ which is $A$-linear and $\GL(E)$-equivariant.
\end{proposition}

\begin{proof}
Suppose first that $\ell(\lambda)=n$.  Let $\mu$ be the partition given by $\mu_1^{\dag}=2n+1-\lambda^{\dag}_2$, $\mu_2^{\dag}=\lambda_1^{\dag}+1=n+1$ and $\mu_i^{\dag}=\lambda_i^{\dag}$ for $i>2$.  Proposition~\ref{o:stepb} provides the following presentation for $\ol{M}_{\lambda}$: 
\begin{displaymath}
\bS_{\mu}(E) \otimes A \to \bS_{\lambda}(E) \otimes A \to \ol{M}_{\lambda} \to 0.
\end{displaymath}
Note that $\bS_{\mu}(E)$ occurs with multiplicity one in $\bS_{\lambda}(E) \otimes A$ by the Littlewood--Richardson rule, and thus does not occur in $\ol{M}_{\lambda}$; it therefore does not occur in $\ol{B}_{\lambda}$ either, since $\ol{M}_{\lambda}$ and $\ol{B}_{\lambda}$ are isomorphic as representations of $\GL(E)$.

Since $\Tor^A_0(B, \bC)=C$, we see from \eqref{o:cring} that $\Tor^A_0(\ol{B}_{\lambda}, \bC)=\bS_{\lambda}(E)$.  It follows that we have a surjection $f \colon A \otimes \bS_{\lambda}(E) \to \ol{B}_{\lambda}$.  Since $\bS_{\mu}(E)$ does not occur in $\ol{B}_{\lambda}$, we see that $f$ induces a surjection $\ol{M}_{\lambda} \to \ol{B}_{\lambda}$.  Since the two are isomorphic as $\GL(E)$ representations and have finite multiplicities, this surjection is an isomorphism.

Now consider the case where $\ell(\lambda)<n$.  Define $\mu$ as above.  Define $\nu$ using the same recipe as for $\mu$ but applied to $\lambda^{\sigma}$; thus $\nu^{\dag}_i=\mu^{\dag}_i$ for $i \ne 2$ and $\nu^{\dag}_2=2n-\lambda_1^{\dag}+1$.  Proposition~\ref{o:stepb} provides the following presentation for $\ol{M}_{\lambda}$: 
\begin{displaymath}
(\bS_{\mu}(E) \otimes A) \oplus (\bS_{\nu}(E) \otimes A) \to (\bS_{\lambda}(E) \otimes A) \oplus (\bS_{\lambda^{\sigma}}(E) \otimes A )\to \ol{M}_{\lambda} \to 0.
\end{displaymath}
Each of $\bS_{\mu}(E)$ and $\bS_{\nu}(E)$ occur with multiplicity one in the middle module by the Littlewood--Richardson rule, and thus neither occurs in $\ol{M}_{\lambda}$; therefore neither occurs in $\ol{B}_{\lambda}$ either.

As $\ol{B}_{\lambda}=B_{\lambda} \oplus B_{\lambda^{\sigma}}$, we see from \eqref{o:cring} that $\Tor^A_0(\ol{B}_{\lambda}, \bC)=\bS_{\lambda}(E) \oplus \bS_{\lambda^{\sigma}}(E)$.  We therefore have a surjection
\begin{displaymath}
f \colon (A \otimes \bS_{\lambda}(E)) \oplus (A \otimes \bS_{\lambda^{\sigma}}(E)) \to \ol{B}_{\lambda}.
\end{displaymath}
Since neither $\bS_{\mu}(E)$ nor $\bS_{\nu}(E)$ occurs in $\ol{B}_{\lambda}$, we see that $f$ induces a surjection $\ol{M}_{\lambda} \to \ol{B}_{\lambda}$.  Since these spaces are isomorphic as $\GL(E)$ representations and have finite multiplicities, this surjection is an isomorphism.
\end{proof}

\begin{remark}
In the second case in the above proof, $\bS_{\mu}(E)$ does not occur in $\bS_{\lambda}(E) \otimes A$ and $\bS_{\nu}(E)$ does not occur in $\bS_{\lambda^{\sigma}}(E) \otimes A$.  It follows that the presentation of $\ol{M}_{\lambda}$ is a direct sum, and so we have a decomposition $\ol{M}_{\lambda}=M_{\lambda} \oplus M_{\lambda^{\sigma}}$.  The argument in the proof shows that $M_{\lambda}=B_{\lambda}$ and $M_{\lambda^{\sigma}}=B_{\lambda^{\sigma}}$.  It would be interesting if the modules $M_{\lambda}$ and $M_{\lambda^{\sigma}}$ could be constructed more directly.
\end{remark}

Combining the above proposition with Proposition~\ref{o:stepb}, we obtain the following corollary.

\begin{corollary}
We have
\begin{displaymath}
\Tor^A_i(B, \bC)=\bigoplus_{i_m(\lambda)=i} \bS_{\lambda}(E) \otimes \bS_{[\ol{\tau}_m(\lambda)]}(V)
\end{displaymath}
as $\GL(E) \times \SO(V)$ representations.
\end{corollary}

Combining this with \eqref{o:tor} shows that
\begin{displaymath}
\rH_i(L^{\lambda}_{\bullet})=\begin{cases}
\bS_{[\ol{\tau}_m(\lambda)]}(V) & \textrm{if $i=i_m(\lambda)$} \\
0 & \textrm{otherwise}
\end{cases}
\end{displaymath}
as representations of $\SO(V)$.  We thus see that the $\bO(V)$-module $\rH_{i_m(\lambda)}(L^{\lambda}_{\bullet})$ is isomorphic to $\bS_{[\tau_m(\lambda)]}(V)$ when restricted to $\SO(V)$, and is therefore either isomorphic to $\bS_{[\tau_m(\lambda)]}(V)$ or $\bS_{[\tau_m(\lambda)^{\sigma}]}(V)$. In fact, it is isomorphic to $\bS_{[\tau_m(\lambda)]}(V)$ by \cite[Theorem 1.9]{wenzl}. This finishes the proof of the main result.

\subsection{The odd case}

Throughout this section, $m=2n+1$ is the dimension of the space $V$.

\Step{a}
Let $Q_0$ be the set of partitions $\lambda$ whose Frobenius coordinates $(a_1, \ldots, a_r \vert b_1, \ldots, b_r)$ satisfy $a_i=b_i$.  Equivalently, $Q_0$ is the set of partitions $\lambda$ such that $\lambda=\lambda^{\dag}$.  This set admits an inductive definition, as follows.  The empty partition belongs to $Q_0$.  A non-empty partition $\lambda$ belongs to $Q_0$ if and only if the number of rows and columns of $\lambda$ are equal, i.e., $\ell(\lambda)=\lambda_1$, and the partition obtained by deleting the first row and column from $\lambda$ belongs to $Q_0$.

Let $\lambda$ be a partition with $\ell(\lambda) \le n$.  We write $(\lambda \vert \mu)$ in place of $(\lambda \mid_n \mu)$.  Define
\begin{displaymath}
\begin{split}
\ol{S}_1(\lambda) &= \{ \textrm{$\mu \in Q_0$ such that $(\lambda \vert \mu)$ is regular} \} \\
\ol{S}_2(\lambda) &= \{ \textrm{partitions $\alpha$ such that $\ol{\tau}_m(\alpha)=\lambda$} \}.
\end{split}
\end{displaymath}

\begin{lemma}
Let $\mu$ be a non-zero partition in $S_1(\lambda)$ and let $\nu$ be the partition obtained by removing the first row and column of $\mu$.  Then $\nu$ also belongs to $S_1(\lambda)$.  Furthermore, let $w$ (resp.\ $w'$) be the element of $W$ such that $\alpha=w \bullet (\lambda \vert \mu)$ (resp.\ $\beta=w' \bullet (\lambda \vert \nu)$) is a partition.  Then $R_{\alpha}$ is defined and we have the following identities:
\begin{displaymath}
\vert R_{\alpha} \vert=2\mu_1+1, \qquad
\alpha \setminus R_{\alpha}=\beta, \qquad
c(R_{\alpha})=\mu_1+\ell(w')-\ell(w).
\end{displaymath}
\end{lemma}

\begin{proof}
Suppose that in applying Bott's algorithm to $(\lambda \vert \mu)$ the number $\mu_1$ moves $r$ places to the left.  Thus, after the first $r$ steps of the algorithm, we reach the sequence
\begin{displaymath}
(\lambda_1, \ldots, \lambda_{n-r}, \mu_1-r, \lambda_{n-r+1}, \ldots, \lambda_n+1, \mu_2, \ldots, \mu_{\ell(\mu)}).
\end{displaymath}
As before, Bott's algorithm on this sequence runs just like the algorithm on $(\lambda \vert \nu)$, and so $(\lambda \vert \nu)$ is regular and $\nu$ belongs to $\ol{S}_1(\lambda)$.  Suppose the algorithm on $(\lambda \vert \nu)$ terminates after $N=\ell(w')$ steps.  Then the algorithm on $(\lambda \vert \mu)$ terminates after $N+r=\ell(w)$ steps, and we have the following formula for $\alpha$:
\begin{displaymath}
\alpha_i=\begin{cases}
\lambda_i & 1 \le i \le n-r \\
\mu_1-r & i=n-r+1 \\
\beta_{i-1}+1 & n-r+2 \le i \le n+\mu_1
\end{cases}
\end{displaymath}
Since $\ell(\alpha)=n+\mu_1$, the border strip $R_{\alpha}$ has $2\mu_1+1$ boxes.  Using Remark~\ref{rmk:striphook}, we see that $R_\alpha$ exists since the box in the $(n-r+1)$th row and the first column has a hook of size $2\ell(\alpha) - m$. Furthermore, $\alpha \setminus R_{\alpha}=\beta$ and $c(R_{\alpha})=\mu_1-r$. Since $r = \ell(w) - \ell(w')$, we are done.
\end{proof}

\begin{proposition}
\label{o2:stepa}
There is a unique bijection $\ol{S}_1(\lambda) \to \ol{S}_2(\lambda)$ under which $\mu$ maps to $\alpha$ if there exists $w \in \fS$ such that $w \bullet (\lambda \vert \mu)=\alpha$; in this case, $\ell(w)+i_m(\alpha)=\tfrac{1}{2}(\vert \mu \vert-\rank(\mu))$ and
\begin{displaymath}
\tau_m(\alpha)=\begin{cases}
\lambda & \textrm{if $\rank(\mu)$ is even} \\
\lambda^{\sigma} & \textrm{if $\rank(\mu)$ is odd.}
\end{cases}
\end{displaymath}
\end{proposition}

\begin{proof}
The proof is exactly the same as for Proposition~\ref{prop:o1:stepa}.
\end{proof}

\Step{b}
Let $E$ be a vector space of dimension at least $n+1$.  Let $X'$ be the partial flag variety of quotients of $E$ of ranks $n+1$ and $n$.  Thus on $X'$ we have vector bundles $\cQ_{n+1}$ and $\cQ_n$ of ranks $n+1$ and $n$, and surjections $E \otimes \cO_{X'} \to \cQ_{n+1} \to \cQ_n$.  Let $\cR^{n+1}$ be the kernel of $E \otimes \cO_{X'} \to \cQ_{n+1}$ and let $\cR^n$ be the kernel of $E \otimes \cO_{X'} \to \cQ_n$.  Put $\cL=\cR^n/\cR^{n+1}$.  Let $X$ be the Grassmannian of rank $n$ quotients of $E$, let $\pi \colon X' \to X$ be the natural map and let $\cR$ and $\cQ$ be the usual bundles on $X$.  Then $\cR^n=\pi^*(\cR)$ and $\cQ^n=\pi^*(\cQ)$.  The space $X'$ is naturally identified with $\bP(\cR)$, with $\cL$ being the universal rank one quotient of $\cR$.

Let $\xi$ be the kernel of $\Sym^2(\cR^n) \to \Sym^2(\cL) = \cL^{\otimes 2}$, let $\eps=\Sym^2(E) \otimes \cO_{X'}$, which contains $\xi$ as a subbundle, and define $\eta$ by the exact sequence
\begin{displaymath}
0 \to \xi \to \eps \to \eta \to 0.
\end{displaymath}
Let $\lambda$ be an admissible partition and let $a=a(\lambda)$ be 0 if $\ell(\lambda) \le n$ and 1 otherwise.  Put
\begin{displaymath}
\cM_{\lambda}=\Sym(\eta) \otimes \bS_{\ol{\lambda}}(\cQ_n) \otimes \cL^{\otimes a}
\end{displaymath}
and $M_{\lambda}=\rH^0(X', \cM_{\lambda})$.  Note that $\rH^0(X', \eps)=A$, and so $M_{\lambda}$ is an $A$-module.  Finally, put
\begin{align*}
W^{\lambda}_i&= \bigoplus_{j \in \bZ} \rH^j(X', \bS_{\ol{\lambda}}(\cQ_n) \otimes \cL^{\otimes a} \otimes \lw^{i+j}(\xi)), \\
\cW^{\lambda}_i &= \bigoplus_{j \in \bZ} \rR^j \pi_*(\bS_{\ol{\lambda}}(\cQ_n) \otimes \cL^{\otimes a} \otimes \lw^{i+j}(\xi)).
\end{align*}
We wish to compute $W^{\lambda}_i$.

\begin{lemma}
\label{o:lem1}
We have
\begin{displaymath}
\cW^{\lambda}_i=\bigoplus_{\mu} \bS_{\ol{\lambda}}(\cQ) \otimes \bS_{\mu}(\cR),
\end{displaymath}
where the sum is over those partitions $\mu$ with $\mu=\mu^{\dag}$, $\rank(\mu)=a \pmod{2}$ and $i=\tfrac{1}{2}(\vert \mu \vert-\rank(\mu))$.
\end{lemma}

\begin{proof}
Since $\cQ_n=\pi^*(\cQ)$, and Schur functors commute with pullback, the projection formula gives
\begin{displaymath}
\cW^{\lambda}_i=\bigoplus_{j \in \bZ} \bS_{\ol{\lambda}}(\cQ) \otimes \rR^j\pi_*(\cL^{\otimes a} \otimes \lw^{i+j}(\xi)).
\end{displaymath}
The result now follows from Proposition~\ref{prop:veronese}.
\end{proof}

\begin{lemma}
\label{o:lem2}
Let $\mu \in S_1(\lambda)$ correspond to $\alpha \in S_2(\lambda)$.  Then
\begin{displaymath}
\rH^i(X, \bS_{\ol{\lambda}}(\cQ) \otimes \bS_{\mu}(\cR))=\begin{cases}
\bS_{\alpha}(E) & \textrm{if $i=\tfrac{1}{2}(\vert \mu \vert-\rank(\mu))-i_m(\alpha)$} \\
0 & \textrm{otherwise.}
\end{cases}
\end{displaymath}
\end{lemma}

\begin{proof}
This follows from Proposition~\ref{o2:stepa} and the Borel--Weil--Bott theorem.
\end{proof}

\begin{lemma}
\label{o:lem3}
We have
\begin{displaymath}
\bigoplus_{j \in \bZ} \rH^j(X, \cW^{\lambda}_{i+j})=\bigoplus_{\alpha} \bS_{\alpha}(E),
\end{displaymath}
where the sum is over those partitions $\alpha$ for which $\tau_m(\alpha)=\lambda$ and $i_m(\alpha)=i$.
\end{lemma}

\begin{proof}
This follows immediately from Lemmas~\ref{o:lem1} and~\ref{o:lem2}, and Proposition~\ref{o2:stepa}.
\end{proof}

\begin{lemma}
\label{o:degen}
The pair $(\pi, \bS_{\ol{\lambda}}(\cQ_n) \otimes \cL^{\otimes a} \otimes \lw^{\bullet}(\xi))$ is degenerate (in the sense of \S \ref{ss:degen}).
\end{lemma}

\begin{proof}
We have
\begin{displaymath}
\bigoplus_{i,j} \rH^i(X, \rR^j\pi_*(\bS_{\ol{\lambda}}(\cQ_n) \otimes \cL^{\otimes a} \otimes \lw^{\bullet}(\xi)))
=\bigoplus_{i,j} \rH^j(X, \cW^{\lambda}_{i+j})
=\bigoplus_{\tau_m(\alpha)=\lambda} \bS_{\alpha}(E).
\end{displaymath}
This is multiplicity-free as a representation of $\GL(E)$, so the criterion of Lemma~\ref{lem:degen} applies.
\end{proof}

\begin{lemma}
We have
\begin{displaymath}
W^{\lambda}_i=\bigoplus_{\alpha} \bS_{\alpha}(E),
\end{displaymath}
where the sum is over those partitions $\alpha$ for which $\tau_m(\alpha)=\lambda$ and $i_m(\alpha)=i$.  In particular, $W^{\lambda}_i=0$ for $i<0$.
\end{lemma}

\begin{proof}
By Lemma~\ref{o:degen}, we have an isomorphism of $\GL(E)$ representations
\begin{displaymath}
W^{\lambda}_i=\bigoplus_{j \in \bZ} \rH^j(X, \cW^{\lambda}_{i+j}).
\end{displaymath}
and so the result follows from Lemma~\ref{o:lem3}
\end{proof}

\begin{proposition}
\label{o2:stepb}
We have
\begin{displaymath}
\Tor^A_i(M_{\lambda}, \bC)=\bigoplus_{\alpha} \bS_{\alpha}(E),
\end{displaymath}
where the sum is over those partitions $\alpha$ for which $\tau_m(\alpha)=\lambda$ and $i_m(\alpha)=i$.
\end{proposition}

\begin{proof}
This follows immediately from the previous lemma and Proposition~\ref{geo}.
\end{proof}

\Step{c}
For an admissible partition $\lambda$, put $B_{\lambda}=\Hom_{\bO(V)}(\bS_{[\lambda]}(V), B)$.  Note that $B_{\lambda}$ is an $A$-module with a compatible action of $\GL(E)$.

\begin{lemma}
The spaces $B_{\lambda} \oplus B_{\lambda^{\sigma}}$ and $M_{\lambda} \oplus M_{\lambda^{\sigma}}$ are isomorphic as representations of $\GL(E)$ and have finite multiplicities.
\end{lemma}

\begin{proof}
Let $M=\bigoplus_\lambda M_{\lambda} \otimes \bS_{[\lambda]}(V)$, where the sum is over admissible partitions $\lambda$.  It is enough to show that $M$ and $B$ are isomorphic as representations of $\GL(E) \times \SO(V)$ and have finite multiplicities.  In fact, it is enough to show that the $\bS_{\theta}(E)$ multiplicity spaces of $M$ and $B$ are isomorphic as representations of $\SO(V)$ and have finite multiplicities.  The proof goes exactly as that of Lemma~\ref{o:iso}.
\end{proof}

\begin{proposition}
Let $\lambda$ be an admissible partition.  We have an isomorphism $M_{\lambda} \to B_{\lambda}$ which is $A$-linear and $\GL(E)$-equivariant.
\end{proposition}

\begin{proof}
Arguing exactly as in the proof of Proposition~\ref{s:stepc} or~\ref{o:stepc}, we obtain a surjection $f \colon M_{\lambda} \to B_{\lambda}$.  Of course, we also have a surjection $f' \colon M_{\lambda^{\sigma}} \to B_{\lambda^{\sigma}}$.  By the previous lemma, $f \oplus f'$ is an isomorphism, and so $f$ and $f'$ are isomorphisms.
\end{proof}

Combining the above proposition with Proposition~\ref{o2:stepb}, we obtain the following corollary.

\begin{corollary}
We have
\begin{displaymath}
\Tor^A_i(B, \bC)=\bigoplus_{i_m(\lambda)=i} \bS_{\lambda}(E) \otimes \bS_{[\tau_m(\lambda)]}(V)
\end{displaymath}
as $\GL(E) \times \bO(V)$ representations.
\end{corollary}

\subsection{Examples}
\label{ss:orthexamples}

We now give a few examples to illustrate the theorem.  

\begin{example}
Suppose $\lambda=(i)$.  Then $L_\bullet^{\lambda}$ is the complex $\Sym^{i-2}{V} \to \Sym^i{V}$, where the differential is multiplication by the symmetric form on $V^*$ treated as an element of $\Sym^2 V$.
\begin{itemize}
\item If $m \ge 2$, or $m=1$ and $i \le 1$, or $m=0$ and $i=0$ then the differential is injective, and $\rH_0(L_\bullet^{\lambda})=\bS_{[i]}(V)$ is a non-zero irreducible representation of $V$.  
\item If $m=1$ and $i \ge 2$, or $m=0$ and $i \ge 3$ or $i=1$, the differential is an isomorphism and all homology of $L_\bullet^{\lambda}$ vanishes.
\item If $m=0$ and $i=2$ then the differential is surjective and $\rH_1(L_\bullet^{\lambda})=\bC$ is the trivial representation of the trivial group $\bO(0)$. \qedhere
\end{itemize}
\end{example}

\begin{example}
Suppose $\lambda=(3,1)$.  Then $L^{\lambda}_\bullet$ is the complex $\bC \to \bS_{(3,1)/(2)}V \to \bS_{(3,1)}V$, where the differentials are multiplication by the symmetric form on $V^*$ treated as an element of $\Sym^2V$.  
\begin{itemize}
\item If $m\ge 3$ then the differential is injective, and $\rH_0(L^{\lambda}_\bullet)=\bS_{[3,1]}(V)$ is an irreducible representation of $V$.  
\item If $m=1,2$ then the complex is exact, and all homology of $L^{\lambda}_\bullet$ vanishes.   
\item Finally when $m=0$ then $\rH_2(L^{\lambda}_\bullet)=\bC$. \qedhere
\end{itemize}
\end{example}

The reader will check easily that in both instances the description of the homology agrees with the rule given by the Weyl group action.

\begin{example}
Let us consider the same situation as in Example~\ref{ex:s3}, i.e., $\lambda=(6,5,4,4,3,3,2)$ and $m=4$.  The modification rule, using border strips, proceeds as follows:
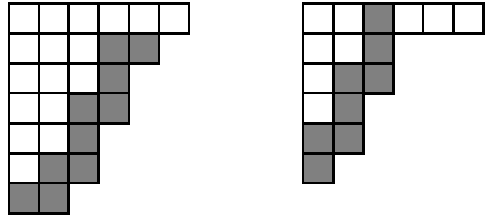
\begin{displaymath}
\ydiagram[*(white)]{6,3,3,2,2,1}*[*(gray)]{6,5,4,4,3,3,2} \qquad\qquad
\ydiagram[*(white)]{2,2,1,1}*[*(gray)]{3,3,3,2,2,1}*[*(white)]{6,3,3,2,2,1}
\end{displaymath}
Starting with $\lambda=\lambda_0$ we remove the border strip $R_0$ of size $2\ell(\lambda)-m=10$.  Doing so we obtain the partition $\lambda_1=(6,3,3,2,2,1)$.  We now are supposed to remove the border strip $R_1$ of size 8.  This border strip is shaded.  However, upon removing this strip we do not have a Young diagram.  It follows that all homology of $L^{\lambda}_\bullet$ vanishes.

In the Weyl group version, this amounts to 
\[
\lambda^\dagger + \rho = (7,7,6,4,2,1) + (-2,-3,-4,\dots) = (5,4,2,-1,-4,-6)
\]
having a nontrivial stabilizer: if $\sigma$ is the transposition that swaps the first and fifth entries, then the stabilizer contains $s_0\sigma s_0$, and this is a non-identity element.
\end{example}

\begin{example}
Suppose $\lambda=(4,4,4,4,3,3,2)$ and $m=4$.  The border strip algorithm runs as follows:
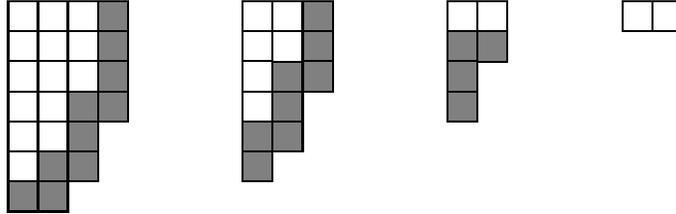
\begin{displaymath}
\ydiagram[*(white)]{3,3,3,2,2,1}*[*(gray)]{4,4,4,4,3,3,2} \qquad\qquad
\ydiagram[*(white)]{2,2,1,1}*[*(gray)]{3,3,3,2,2,1} \qquad\qquad
\ydiagram[*(white)]{2}*[*(gray)]{2,2,1,1} \qquad\qquad
\ydiagram[*(white)]{2}
\end{displaymath}
We have removed three border strips $R_0$, $R_1$, $R_2$.  Thus according to rule (D3) of \S \ref{o:modrule}, $\tau_4(\lambda)$ is not the final partition $(2)$, but its conjugate, i.e., $\tau_4(\lambda)=(3,1)$.  We have
\begin{displaymath}
i_4(\lambda)=(c(R_0)-1)+(c(R_1)-1)+(c(R_2)-1)=3+2+1=6.
\end{displaymath}
We thus see that $\rH_i(L^{\lambda}_\bullet)=0$ if $i \ne 6$ and $\rH_6(L^{\lambda}_\bullet)=\bS_{[3,1]}(V)$.

Now we illustrate the modification rule using the Weyl group action. We write $\alpha \xrightarrow{s_i} \beta$ if $\beta = s_i(\alpha)$. The idea for getting the Weyl group element is to apply $s_0$ if sum of the first two column lengths is too big, then sort the result, and repeat as necessary. We start with $\lambda^\dagger + \rho = (7,7,6,4) + (-2,-3,-4,-5,\dots)$:
\begin{align*}
(5,4,2,-1) &\xrightarrow{s_0} (-4,-5,2,-1)
\xrightarrow{s_2} (-4,2,-5,-1)\\
&\xrightarrow{s_3} (-4,2,-1,-5)
\xrightarrow{s_1} (2,-4,-1,-5)\\
&\xrightarrow{s_2} (2,-1,-4,-5)
\xrightarrow{s_0} (1,-2,-4,-5).
\end{align*}
Subtracting $\rho$ from the result, we get $(3,1) = (2,1,1)^\dagger$.
\end{example}

\section{General linear groups}

\subsection{Representations of $\GL(V)$}

Let $V$ be a vector space of dimension $n$.  The irreducible representations of $\GL(V)$ are indexed by pairs of partitions $(\lambda, \lambda')$ such that $\ell(\lambda)+\ell(\lambda') \le n$ (see \cite[\S 1]{koike} for more details).
We call such pairs {\bf admissible}.  Given an admissible pair $(\lambda, \lambda')$, we denote by $\bS_{[\lambda,\lambda']}(V)$ the corresponding irreducible representation of $\GL(V)$.  Identifying weights of $\GL(V)$ with elements of $\bZ^n$, the representation $\bS_{[\lambda,\lambda']}(V)$ is the irreducible with highest weight $(\lambda_1, \ldots, \lambda_r, 0, \ldots, 0, -\lambda'_s, \ldots, -\lambda'_1)$, where $r=\ell(\lambda)$ and $s=\ell(\lambda')$.  The representation $\bS_{[\lambda, 0]}(V)$ is the usual Schur functor $\bS_{\lambda}(V)$, while the representation $\bS_{[0, \lambda]}$ is its dual $\bS_{\lambda}(V)^*$.

\subsection{The Littlewood complex}
\label{ss:glcx}

Let $E$ and $E'$ be vector spaces.  Put $U=E \otimes E'$, $A=\Sym(U)$ and $B=\Sym((E \otimes V) \oplus (E' \otimes V^*))$.  Let $U \subset B$ be the inclusion given by
\begin{displaymath}
E \otimes E' \subset (E \otimes V) \otimes (E' \otimes V^*) \subset \Sym^2(E \otimes V \oplus E' \otimes V^*),
\end{displaymath}
where the first inclusion is multiplication with the identity element of $V \otimes V^*$.  This inclusion defines an algebra homomorphism $A \to B$.  Let $C=B \otimes_A \bC$; this is the quotient of $B$ by the ideal generated by $U$.  We have maps
\begin{displaymath}
\Spec(C) \to \Spec(B) \to \Spec(A).
\end{displaymath}
We have a natural identification of $\Spec(B)$ with the space $\Hom(E, V^*) \times \Hom(E', V)$ of pairs of maps $(\varphi \colon E \to V^*, \psi \colon E' \to V)$.  The space $\Spec(A)$ is naturally identified with the space $(E \otimes E')^*$ of bilinear forms on $E \times E'$.  The map $\Spec(B) \to \Spec(A)$ takes a pair of maps $(\varphi, \psi)$ to the form $(\varphi \otimes \psi)^* \omega$, where $\omega \colon V \otimes V^* \to \bC$ is the trace map.  The space $\Spec(C)$, which we call the {\bf Littlewood variety}, is the scheme-theoretic fiber of this map above 0, i.e., it consists of those pairs of maps $(\varphi, \psi)$ such that $(\varphi \otimes \psi)^* \omega=0$. 

\begin{remark}
One can modify the definitions of the rings $A$, $B$ and $C$ by replacing $E$ with its dual everywhere.  The space $\Spec(A)$ is then identified with $\Hom(E', E)$, while $\Spec(B)$ is identified with the set of pairs of maps $(\varphi \colon V \to E, \psi \colon E' \to V)$.  The map $\Spec(B) \to \Spec(A)$ takes $(\varphi, \psi)$ to $\varphi \psi$.  The space $\Spec(C)$ consists of those pairs $(\varphi, \psi)$ such that $\varphi \psi=0$; thus $\Spec(C)$ is the space of complexes of the form $E' \to V \to E$.
\end{remark}

Let $\rK_{\bullet}(E, E')=B \otimes \lw^{\bullet}{U}$ be the Koszul complex of the Littlewood variety.  We can decompose this complex under the action of $\GL(E) \times \GL(E')$:
\begin{displaymath}
\rK_{\bullet}(E, E')=\bigoplus_{\substack{\ell(\lambda) \le \dim{E}\\ \ell(\lambda') \le \dim{E'}}} \bS_{\lambda}(E) \otimes \bS_{\lambda'}(E') \otimes L^{\lambda,\lambda'}_{\bullet}.
\end{displaymath}
The complex $L^{\lambda,\lambda'}_{\bullet}$ is the {\bf Littlewood complex}, and is independent of $E$ and $E'$ (so long as $\dim{E} \ge \ell(\lambda)$ and $\dim{E'} \ge \ell(\lambda')$).  By \cite[Theorem 3.3]{brylinski}, its zeroth homology is
\begin{equation}
\label{gl:Hzero}
\rH_0(L^{\lambda,\lambda'}_{\bullet}) = \begin{cases}
\bS_{[\lambda,\lambda']}(V) & \textrm{if $(\lambda, \lambda')$ is admissible} \\
0 & \textrm{otherwise.}
\end{cases}
\end{equation}
By Lemma~\ref{lem:koszul}, we have $\rH_i(\rK_{\bullet})=\Tor^A_i(B, \bC)$, and so we have a decomposition
\begin{equation}
\label{gl:tor}
\Tor^A_i(B, \bC)=\bigoplus_{\substack{\ell(\lambda) \le \dim{E}\\ \ell(\lambda') \le \dim{E'}}} \bS_{\lambda}(E) \otimes \bS_{\lambda'}(E') \otimes \rH_i(L^{\lambda,\lambda'}_{\bullet}).
\end{equation}
Applied to $i=0$, we obtain
\begin{equation}
\label{gl:cring}
C = \bigoplus_{\textrm{admissible $(\lambda, \lambda')$}} \bS_{\lambda}(E) \otimes \bS_{\lambda'}(E') \otimes \bS_{[\lambda,\lambda']}(V).
\end{equation}

\subsection{A special case of the main theorem}

Our main theorem computes the homology of the complex $L^{\lambda}_{\bullet}$.  We now formulate and prove the theorem in a particularly simple case.  We mention this here only because it is worthwhile to know; the argument is not needed to prove the main theorem.

\begin{proposition}
Suppose $(\lambda, \lambda')$ is admissible.  Then
\begin{displaymath}
\rH_i(L^{\lambda, \lambda'}_{\bullet}) = \begin{cases}
\bS_{[\lambda, \lambda']}(V) & \textrm{if $i=0$} \\
0 & \textrm{otherwise.}
\end{cases}
\end{displaymath}
\end{proposition}

\begin{proof}
Choose $E, E'$ so that $\dim E = \ell(\lambda)$ and $\dim E' =  \ell(\lambda')$.  By Lemma~\ref{lem:A:CI} below, $\rK_{\bullet}(E, E')$ has no higher homology.  It follows that $L^{\lambda, \lambda'}_{\bullet}$ does not either.  The computation of $\rH_0(L^{\lambda, \lambda'}_{\bullet})$ is given in \eqref{gl:Hzero}.
\end{proof}

\begin{lemma} \label{lem:A:CI}
Suppose $\dim{E} + \dim E' \le n$.  Then $U \subset B$ is spanned by a regular sequence.
\end{lemma}

\begin{proof}
It suffices to show that $\dim{\Spec(C)}=\dim{\Spec(B)}-\dim{U}$.  Put $d=\dim{E}$ and $d' = \dim E'$.  Observe that the locus of $(\phi, \psi)$ in $\Spec(C)$ where $\varphi$ is injective and $\psi$ is surjective is open. Let ${\bf Fl}(d,d+d',V)$ be the variety of partial flags $W_d \subset W_{d+d'} \subset V$, where the subscript indicates the dimension of the subspace. This variety comes with a tautological partial flag $\cR_d \subset \cR_{d+d'} \subset V \otimes \cO_{{\bf Fl}(d,d+d',V)}$. There is a natural birational map from the total space of $\cH om(E,\cR_d) \oplus \cH om(\cR_{d+d'} / \cR_d, E')$ to $\Spec(C)$, and thus $\Spec(C)$ has dimension $n(d+d')-dd'$.  As $\dim{\Spec(B)}=n(d+d')$ and $\dim{U}=dd'$, the result follows.
\end{proof}

\subsection{The modification rule} \label{ss:gl-modrule}

We now associate to a pair of partitions $(\lambda, \lambda')$ two quantities, $i_n(\lambda, \lambda')$ and $\tau_n(\lambda, \lambda')$.  As usual, we give two equivalent definitions.

We begin with Koike's definition of the Weyl group definition from the discussion preceding \cite[Prop.~2.2]{koike}. First, consider $-(\lambda'^\dagger)^{\rm op}$, which is the sequence obtained by taking $\lambda'^\dagger$ and reversing and negating its entries. We add $n$ to each entry, and call the result $\sigma(\lambda')$. For example, if $\lambda' = (3,2,2,1,1)$ and $n=4$, then $-(\lambda'^\dagger)^{\rm op} = (-1,-3,-5)$ and $\sigma(\lambda') = (3,1,-1)$.  Note that if $\mu$ is any weakly decreasing sequence of nonnegative integers with $n \ge \mu_1$, then it makes sense to reverse this procedure, and so we can define $\sigma^{-1}(\mu)$. Now set $\alpha = (\sigma(\lambda') \mid \lambda^\dagger)$ and $\rho = (\dots, 3, 2, 1 \mid 0, -1, -2, \dots)$.  We let $W'$ be the group of permutations on the index set $\bZ$ of the coordinates of $\rho$ and $\alpha$ which differ from the identity permutation in only finitely many places. Given a permutation $w \in W'$, we define $w \bullet \alpha = w(\alpha + \rho) - \rho$ as usual. One of two possibilities occurs:
\begin{itemize}
\item There exists a unique element $w \in W'$ so that $w \bullet \alpha = (\beta' \mid \beta)$ is weakly decreasing. In this case, we set $i_n(\lambda, \lambda') = \ell(w)$ and $\tau_n(\lambda, \lambda') = (\beta, \sigma^{-1}(\beta'))$.
\item There exists a non-identity element $w \in W'$ such that $w \bullet \alpha = \alpha$. In this case, we put $i_n(\lambda, \lambda') = \infty$ and leave $\tau_n(\lambda, \lambda')$ undefined.
\end{itemize}

We now give a modified (but equivalent) description of this Weyl group action. The group $\fS \times \fS$ acts on the set $\cU \times \cU$ via the $\bullet$ action.  We define a new involution $t$ of $\cU \times \cU$ by
\begin{displaymath}
t \bullet ((a_1, a_2, \ldots), (b_1, b_2, \ldots))
= ((n+1-b_1, a_2, \ldots), (n+1-a_1, b_2, \ldots))
\end{displaymath}
Let $W$ be the subgroup of $\Aut(\cU \times \cU)$ generated by $\fS \times \fS$ and $t$.  Then $W$ is isomorphic to an infinite symmetric group, and comes equipped with a length function $\ell \colon W \to \bZ_{\ge 0}$ with respect to its set of generators, as in \eqref{eqn:word}.  Given a pair of partitions $(\lambda, \lambda') \in \cU \times \cU$, exactly one of the following two possibilities hold:
\begin{itemize}
\item There exists a unique element $w \in W$ such that $w \bullet (\lambda, \lambda')^{\dag}=(\mu, \mu')^{\dag}$ is a pair of partitions and $(\mu, \mu')$ is admissible.  We then put $i_n(\lambda, \lambda')=\ell(w)$ and $\tau_n(\lambda, \lambda')=(\mu, \mu')$.  (The notation $(\lambda, \lambda')^{\dag}$ simply means $(\lambda^{\dag}, (\lambda')^{\dag})$.)
\item There exists a non-identity element $w \in W$ such that $w \bullet (\lambda, \lambda')^{\dag}=(\lambda, \lambda')^{\dag}$.  We then put $i_n(\lambda, \lambda')=\infty$ and leave $\tau_n(\lambda, \lambda')$ undefined.
\end{itemize}
As always, if $(\lambda, \lambda')$ is already admissible then we are in the first case, and $i_n(\lambda, \lambda')=0$ and $\tau_n(\lambda, \lambda')=(\lambda, \lambda')$.

~

We now give the border strip definition, which we could not find in the literature.  If $(\lambda, \lambda')$ is admissible then we define $i_n(\lambda, \lambda')=0$ and $\tau_n(\lambda, \lambda')=(\lambda, \lambda')$.  Assume now that $(\lambda, \lambda')$ is not admissible.  Let $R_{\lambda}$ (resp.\ $R_{\lambda'}$) be the border strip of length $\ell(\lambda)+\ell(\lambda')-n-1$ starting in the first box of the final row of $\lambda$ (resp.\ $\lambda')$, if it exists.  If both $R_{\lambda}$ and $R_{\lambda'}$ exist and are non-empty and both $\lambda \setminus R_{\lambda}$ and $\lambda' \setminus R_{\lambda'}$ are partitions, we put
\begin{displaymath}
i_n(\lambda, \lambda')=c(R_{\lambda})+c(R_{\lambda'})-1+i_n(\lambda \setminus R_{\lambda}, \lambda' \setminus R_{\lambda'})
\end{displaymath}
and $\tau_n(\lambda, \lambda')=\tau_n(\lambda \setminus R_{\lambda}, \lambda' \setminus R_{\lambda'})$.  Otherwise, we put $i_n(\lambda, \lambda')=\infty$ and leave $\tau_n(\lambda, \lambda')$ undefined.

\begin{proposition}
The above two definitions agree.
\end{proposition}

\begin{proof}
Consider the diagram
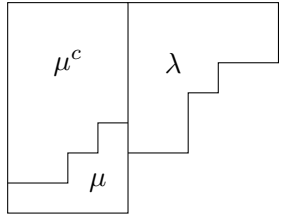
\[
\begin{tikzpicture}[scale=.4]
\draw (0,0) -- (4,0) -- (4,-7) -- (0,-7) -- cycle;
\path (2,-2) node {$\mu^c$};
\draw (4,0) -- (9,0) -- (9,-2) -- (7,-2) -- (7,-3) -- (6,-3) -- (6,-5) -- (4,-5);
\path (5.5,-2) node {$\lambda$};
\draw (4,-4) -- (3,-4) -- (3,-5) -- (2,-5) -- (2,-6) -- (0,-6);
\path (3,-6) node {$\mu$}; 
\end{tikzpicture}
\]
where $\mu^c$ is the complement of $\mu$ in a $n \times \mu_1$ rectangle (we are allowing the possibility that $\ell(\mu) > n$, in which case, we need to consider negative column lengths, but for the purposes of explanation, we assume $\ell(\mu) \le n$). According to Koike's rule, we are supposed to apply Bott's algorithm to the columns of the union of $\mu^c$ and $\lambda$. 

We describe two procedures:

\begin{enumerate}
\item Start with a shape that is a single column of length $d$ ($d$ could be negative) union a partition $\lambda$ with $\ell(\lambda) > d$ and apply Bott's algorithm. Then we end up with a single column of length $\ell(\lambda)-1$ union the shape obtained by removing a border strip of size $\ell(\lambda) - d - 1$ if that is possible, and 0 otherwise. 

\item Dually, start with a shape $\nu$ union a column of length $e$ with $e > \nu^\dagger_{\nu_1}$ and apply Bott's algorithm. Then we get the shape $\eta$ union a column of length $\nu^\dagger_{\nu_1} + 1$ where $\eta$ is obtained from $\nu$ by adding a border strip of length $e - \nu^\dagger_{\nu_1} - 1$ starting from the bottom box in the last column of $\nu$ if it exists, and 0 otherwise.
\end{enumerate}

Now go back to our shape, which is $(\mu^c, \lambda)$. Apply (1) with the column being the last column of $\mu^c$ so that $d = n - \ell(\mu)$. Then the last column of $\mu^c$ becomes $\ell(\lambda) - 1$ and we have removed a border strip of size $\ell(\lambda) + \ell(\mu) - n - 1$ from $\lambda$. In the process, we used $c(\lambda)$ simple reflections. Now apply (2) with $\nu$ being $\mu^c$ minus its last column and $e = \ell(\lambda) - 1$. The result is the shape obtained by adding a border strip (starting from the right, not left) of length $\ell(\lambda) - \nu^\dagger_{\nu_1}$ to $\nu$ union a column of length $\nu^\dagger_{\nu_1} + 1$. We can also describe this shape as follows: we added to $\mu^c$ a border strip (starting from the right) of size $\ell(\lambda) + \ell(\mu) - n - 1$. In this second step, we have used $c(\mu)-1$ simple reflections. 

Finally, note that adding a border strip, starting from the right, to $\mu^c$ is equivalent to removing a border strip from $\mu$ in the usual sense.
\end{proof}

\subsection{The main theorem}

Our main theorem is the following:

\begin{theorem}
\label{gl:mainthm}
For a pair of partitions $(\lambda, \lambda')$ and an integer $i$, we have
\begin{displaymath}
\rH_i(L^{\lambda,\lambda'}_{\bullet})=\begin{cases}
\bS_{[\tau_n(\lambda, \lambda')]}(V) & \textrm{if $i=i_n(\lambda, \lambda')$} \\
0 & \textrm{otherwise.}
\end{cases}
\end{displaymath}
In particular, if $i_n(\lambda, \lambda')=\infty$ then $L^{\lambda,\lambda'}_{\bullet}$ is exact.
\end{theorem}

\begin{remark}
Consider the coordinate ring $R$ of rank $\le n$ matrices; identifying $E$ with its dual, this is the quotient of $A$ by the ideal generated by $(n+1) \times (n+1)$ minors.  The resolution of $R$ over $A$ is known, see \cite[\S 6.1]{weyman} or \cite{lascoux}.  On the other hand, $R$ is the $\GL(V)$-invariant part of $B$, and so the above theorem, combined with \eqref{gl:tor}, shows that $\bS_{\lambda}(E) \otimes \bS_{\lambda'}(E')$ appears in its resolution if and only if $\tau_n(\lambda,\lambda') = \emptyset$. Thus the modification rule gives an alternative description of the resolution of $R$. It is a pleasant combinatorial exercise to show directly that these two descriptions are equivalent. As in previous situations, the modification rule is more complicated, but has the advantage that it readily generalizes to our situation.
\end{remark}

The proof will take the rest of this section, and will follow the three-step plan given in \S\ref{ss:proofoverview}.

\Step{a}
Fix integers $a$ and $b$ with $n=a+b$.  Let $(\lambda, \lambda')$ be a pair of partitions with $\ell(\lambda) \le a$ and $\ell(\lambda') \le b$.  We write $(\lambda \vert \mu)$ in place of $(\lambda \mid_a \mu)$ and $(\lambda' \vert \mu')$ in place of $(\lambda' \mid_b \mu')$.  Define
\begin{displaymath}
\begin{split}
S_1(\lambda, \lambda') &= \{ \textrm{partitions $\mu$ such that $(\lambda \vert \mu)$ and $(\lambda' \vert \mu^{\dag})$ are regular} \} \\
S_2(\lambda, \lambda') &= \{ \textrm{pairs of partitions $(\alpha, \alpha')$ such that $\tau_n(\alpha, \alpha')=(\lambda, \lambda')$} \}.
\end{split}
\end{displaymath}

\begin{lemma}
Let $\mu$ be a non-zero partition in $S_1(\lambda, \lambda')$ and let $\nu$ be the partition obtained by removing the first row and column of $\mu$.  Then $\nu$ also belongs to $S_1(\lambda,\lambda')$.  Furthermore, let $(w_1, w_1')$ (resp.\ $(w_2, w_2')$) be elements of $W$ such that $\alpha=w_1 \bullet (\lambda \vert \mu)$ and $\alpha'=w'_1 \bullet (\lambda \vert \mu^{\dag})$ (resp.\ $\beta=w_2 \bullet (\lambda \vert \nu)$ and $\beta'=w_2' \bullet (\lambda \vert \nu^{\dag})$) are partitions.  Then $R_{\alpha}$ and $R_{\alpha'}$ are defined and we have the identities
\begin{displaymath}
\vert R_{\alpha} \vert=\vert R_{\alpha'} \vert=\ell(\mu)+\mu_1-1, \qquad
\alpha \setminus R_{\alpha}=\beta, \qquad
\alpha' \setminus R_{\alpha'}=\beta'
\end{displaymath}
and
\begin{displaymath}
c(R_{\alpha})+c(R_{\alpha'})=\ell(\mu)+\mu_1+\ell(w_2)+\ell(w_2')-\ell(w_1)-\ell(w_1').
\end{displaymath}
\end{lemma}

\begin{proof}
Reasoning as in the proof of Lemma~\ref{lem:s1}, we see that $\nu$ belongs to $S_1(\lambda, \lambda')$.  Suppose that in applying Bott's algorithm to $(\lambda \vert \mu)$ (resp. $(\lambda' \vert \mu^{\dag})$) the number $\mu_1$ (resp.\ $\mu^{\dag}_1$) moves $r$ (resp.\ $s$) places to the left.  Let $N=\ell(w_2)$ (resp.\ $N'=\ell(w_2')$) be the number of steps in Bott's algorithm applied to $(\lambda \vert \nu)$ (resp.\ $(\lambda \vert \nu^{\dag})$).  Then, as in the proof of Lemma~\ref{lem:s1}, we can trace Bott's algorithm on $(\lambda \vert \mu)$ and $(\lambda' \vert \mu^{\dag})$.  We find
\begin{align*}
\alpha_i&=\begin{cases}
\lambda_i & 1 \le i \le a-r \\
\mu_1-r & i=a-r+1 \\
\beta_{i-1}+1 & a-r+1 \le i \le a+\ell(\mu)
\end{cases}\\
\alpha'_i&=\begin{cases}
\lambda'_i & 1 \le i \le b-s \\
\mu^{\dag}_1-s & i=b-s+1 \\
\beta'_{i-1}+1 & b-s+2 \le i \le b+\mu_1.
\end{cases}
\end{align*}
We now examine the border strips used in the modification rule for $(\alpha, \alpha')$.  Both $R_{\alpha}$ and $R_{\alpha'}$ have $\ell(\alpha)+\ell(\alpha')-n-1=\ell(\mu)+\mu_1-1$ boxes. Using Remark~\ref{rmk:striphook}, we see that $R_\alpha$ (resp. $R_{\alpha'}$) exists since the box in $(a-r+1)$th (resp. $(b-s+1)$th) row and the first column has a hook of size $\ell(\mu)+\mu_1-1$. Furthermore, $\alpha \setminus R_{\alpha}=\beta$ and $\alpha' \setminus R_{\alpha'}=\beta'$ and
\begin{align*}
c(R_{\alpha}) &=\ell(\mu)+\mu_1-\ell(\alpha)+a-r, \\
c(R_{\alpha'}) &=\ell(\mu)+\mu_1-\ell(\alpha')+b-s,
\end{align*}
which completes the proof.
\end{proof}

\begin{proposition}
\label{gl:stepa}
There is a unique bijection $S_1(\lambda, \lambda') \to S_2(\lambda, \lambda')$ under which $\mu$ maps to $(\alpha, \alpha')$ if there exists $w$ and $w'$ in $\fS$ such that $w \bullet (\lambda \vert \mu)=\alpha$ and $w' \bullet (\lambda' \vert \mu^{\dag})=\alpha'$; in this case, $\ell(w)+\ell(w')+i_n(\alpha, \alpha')=\vert \mu \vert$.
\end{proposition}

\begin{proof}
The proof is essentially the same as that of Proposition~\ref{s:stepa}.
\end{proof}

\Step{b}
Let $E$ and $E'$ be vector spaces of dimensions at least $a$ and $b$, respectively.  Let $X$ (resp.\ $X'$) be the Grassmannian of rank $a$ (resp.\ $b$) quotients of $E$ (resp.\ $E'$).  Let $\cR$ and $\cQ$ (resp.\ $\cR'$ and $\cQ'$) be the tautological bundles on $X$ (resp.\ $X'$) as in \eqref{eqn:taut}; we regard all four as bundles on $X \times X'$.  Put $\eps=E \otimes E' \otimes \cO_{X \times X'}$, $\xi=\cR \otimes \cR'$ and define $\eta$ by the exact sequence
\begin{displaymath}
0 \to \xi \to \eps \to \eta \to 0.
\end{displaymath}
Finally, for a pair of partitions $(\lambda, \lambda')$ with $\ell(\lambda) \le a$ and $\ell(\lambda') \le b$, put
\begin{displaymath}
\cM_{\lambda,\lambda'}=\Sym(\eta) \otimes \bS_{\lambda}(\cQ) \otimes \bS_{\lambda'}(\cQ')
\end{displaymath}
and $M_{\lambda, \lambda'}=\rH^0(X \times X', \cM_{\lambda, \lambda'})$.  Note that $A=\rH^0(X, \Sym(\eps))$, and so $M_{\lambda,\lambda'}$ is an $A$-module.

\begin{lemma} \label{lem:gl_vbundle}
Let $\lambda$ and $\lambda'$ be partitions with $\ell(\lambda) \le a$ and $\ell(\lambda') \le b$.  Let $\nu \in S_1(\lambda, \lambda')$ correspond to $(\alpha, \alpha')$ in $S_2(\lambda, \lambda')$.  Let $\cV$ be the vector bundle on $X \times X'$ given by
\begin{displaymath}
\cV=\bS_{\lambda}(\cQ) \otimes \bS_{\lambda'}(\cQ') \otimes \bS_{\nu}(\cR) \otimes \bS_{\nu^{\dag}}(\cR').
\end{displaymath}
Then
\begin{displaymath}
\rH^i(X \times X', \cV) = \begin{cases}
\bS_{\alpha}(E) \otimes \bS_{\alpha'}(E') & \textrm{if $i=\vert \nu \vert-i_n(\alpha, \alpha')$} \\
0 & \textrm{otherwise.}
\end{cases}
\end{displaymath}
\end{lemma}

\begin{proof}
The K\"unneth formula shows that
\begin{displaymath}
\rH^{\bullet}(X \times X', \cV)=\rH^{\bullet}(X, \bS_{\lambda}(\cQ) \otimes \bS_{\nu}(\cR)) \otimes \rH^{\bullet}(X', \bS_{\lambda'}(\cQ') \otimes \bS_{\nu^{\dag}}(\cR')).
\end{displaymath}
By definition, $(\lambda \vert \nu)$ and $(\lambda' \vert \nu^{\dag})$ are both regular, and so we can find $w$ and $w'$ in $\fS$ such that $\alpha=w \bullet (\lambda \vert \nu)$ and $\alpha'=w \bullet (\lambda' \vert \nu^{\dag})$ are partitions. By the Borel--Weil--Bott theorem, we get 
\begin{align*}
\rH^i(X, \bS_{\lambda}(\cQ) \otimes \bS_{\nu}(\cR)) &=\begin{cases}
\bS_{\alpha}(E) & \textrm{if $i=\ell(w)$} \\
0 & \textrm{otherwise}
\end{cases}\\
\rH^i(X', \bS_{\lambda}(\cQ') \otimes \bS_{\nu^{\dag}}(\cR')) &=\begin{cases}
\bS_{\alpha'}(E') & \textrm{if $i=\ell(w')$} \\
0 & \textrm{otherwise.}
\end{cases}
\end{align*}
By Proposition~\ref{gl:stepa}, we have $\ell(w)+\ell(w')=\vert \nu \vert-i_n(\alpha, \alpha')$, which completes the proof.
\end{proof}

\begin{lemma} \label{lem:gl_coh}
Let $\lambda$ and $\lambda'$ be partitions with $\ell(\lambda) \le a$ and $\ell(\lambda') \le b$ and let $i$ be an integer.  We have
\begin{displaymath}
\bigoplus_{j \in \bZ} \rH^j(X \times X', \lw^{i+j}(\xi) \otimes \bS_{\lambda}(\cQ) \otimes \bS_{\lambda'}(\cQ')) = \bigoplus_{(\alpha, \alpha')} \bS_{\alpha}(E) \otimes \bS_{\alpha'}(E'),
\end{displaymath}
where the sum is over pairs $(\alpha, \alpha')$ with $i_n(\alpha, \alpha')=i$ and $\tau_n(\alpha, \alpha')=(\lambda, \lambda')$.  In particular, when $i<0$ the left side above vanishes.
\end{lemma}

\begin{proof}
We have
\begin{displaymath}
\lw^{i+j}(\xi)=\lw^{i+j}(\cR \otimes \cR')=\bigoplus_{\vert \nu \vert=i+j} \bS_{\nu}(\cR) \otimes \bS_{\nu^{\dag}}(\cR'),
\end{displaymath}
and so
\begin{displaymath}
\begin{split}
& \bigoplus_{j \in \bZ} \rH^j(X \times X', \lw^{i+j}(\xi) \otimes \bS_{\lambda}(\cQ) \otimes
\bS_{\lambda'}(\cQ')) \\
=& \bigoplus_{\nu} \rH^{\vert \nu \vert-i}(X \times X', \bS_{\lambda}(\cQ) \otimes \bS_{\lambda'}(\cQ') \otimes \bS_{\nu}(\cR) \otimes \bS_{\nu^{\dag}}(\cR')).
\end{split}
\end{displaymath}
Of course, we only need to sum over $\nu \in S_1(\lambda, \lambda')$.  The result now follows from Lemma~\ref{lem:gl_vbundle}.
\end{proof}

\begin{proposition}
\label{gl:stepb}
Let $\lambda$ and $\lambda'$ be partitions with $\ell(\lambda) \le a$ and $\ell(\lambda') \le b$.  We have
\begin{displaymath}
\Tor^A_i(M_{\lambda, \lambda'}, \bC)=\bigoplus_{(\alpha, \alpha')} \bS_{\alpha}(E) \otimes \bS_{\alpha'}(E'),
\end{displaymath}
where the sum is over pairs $(\alpha, \alpha')$ with $i_n(\alpha, \alpha')=i$ and $\tau_n(\alpha, \alpha')=(\lambda, \lambda')$.
\end{proposition}

\begin{proof}
This follows immediately from the Lemma~\ref{lem:gl_coh} and Proposition~\ref{geo}.
\end{proof}

\begin{lemma}
\label{gl:m-ind}
The module $M_{\lambda,\lambda'}$ is independent of the choice of $a$ and $b$, provided that $\ell(\lambda) \le a$ and $\ell(\lambda') \le b$.
\end{lemma}

\begin{proof}
Applying Proposition~\ref{gl:stepb} with $i=0$ and $i=1$ shows that we have a presentation
\begin{displaymath}
A \otimes \bS_{(\lambda, 1^d)}(E) \otimes \bS_{(\lambda', 1^d)}(E') \to A \otimes \bS_{\lambda}(E) \otimes \bS_{\lambda'}(E') \to M_{\lambda,\lambda'} \to 0,
\end{displaymath}
where $d=n+1-\ell(\lambda)-\ell(\mu)$.  The left map is unique up to scalar multiple, since $\bS_{(\lambda,1^d)}(E) \otimes \bS_{(\lambda',1^d)}(E')$ occurs with multiplicity one in the middle group, and so the lemma follows.
\end{proof}

We thus have a well-defined $A$-module $M_{\lambda, \lambda'}$ for any admissible pair $(\lambda, \lambda')$.

\Step{c}
For an admissible pair $(\lambda, \lambda')$, put
\begin{displaymath}
B_{\lambda,\lambda'}=\Hom_{\GL(V)}(\bS_{[\lambda,\lambda']}(V), B).
\end{displaymath}
Note that $B_{\lambda,\lambda'}$ is an $A$-module and has a compatible action of $\GL(E) \times \GL(E')$.  Our goal is to show that $B_{\lambda,\lambda'}$ is isomorphic to $M_{\lambda,\lambda'}$.

\begin{lemma}
Let $(\lambda, \lambda')$ be an admissible pair.  Then the spaces $B_{\lambda,\lambda'}$ and $M_{\lambda,\lambda'}$ are isomorphic as representations of $\GL(E) \times \GL(E')$, and have finite multiplicities.
\end{lemma}

\begin{proof}
As we have done before, put
\begin{displaymath}
M=\bigoplus_{\textrm{admissible $(\lambda, \lambda')$}} M_{\lambda,\lambda'} \otimes \bS_{[\lambda,\lambda']}(V).
\end{displaymath}
It is enough to show that $M$ and $B$ are isomorphic as representations of $\GL(E) \times \GL(E') \times \GL(V)$ and have finite multiplicities.  In fact, it is enough to show that the $\bS_{\theta}(E) \otimes \bS_{\theta'}(E')$ multiplicity spaces of $M$ and $B$ are isomorphic as representations of $\GL(V)$ and have finite multiplicities.  This is what we do.

The $\bS_{\theta}(E) \otimes \bS_{\theta'}(E)$ multiplicity space of $B$ is $\bS_{\theta}(V) \otimes \bS_{\theta'}(V^*)$.  The decomposition of this in the representation ring of $\GL(V)$ can be computed using \cite[Thm.~2.4]{koike}.  The result is
\begin{displaymath}
\sum_{\nu,\nu',\mu} (-1)^{i_n(\nu, \nu')} c^{\theta}_{\nu, \mu} c^{\theta'}_{\nu', \mu} [\bS_{[\tau_n(\nu, \nu')]}(V)].
\end{displaymath}
(In the notation of \cite{koike}, we are computing $M^{[\nu,\nu']}_{[\theta,0],[0,\theta']}$.  These zeros lead to the massive simplification of the general formula given there.)  Note that for fixed $(\theta, \theta')$ there are only finitely many values for $(\nu, \nu', \mu)$ which make the product of Littlewood--Richardson coefficients non-zero, which establishes finiteness of the multiplicities.  Now, we have an equality
\begin{displaymath}
[M_{\lambda,\lambda'}]=[A] \sum_{i \ge 0} (-1)^i [\Tor^A_i(M_{\lambda, \lambda'}, \bC)]
\end{displaymath}
in the representation ring of $\GL(E)$.  Applying Proposition~\ref{gl:stepb}, we find
\begin{displaymath}
\sum_{i \ge 0} (-1)^i [\Tor^A_i(M_{\lambda, \lambda'}, \bC)] = \sum_{\substack{\nu, \nu'\\ \tau_n(\nu, \nu')=(\lambda, \lambda')}} (-1)^{i_n(\nu, \nu')} [\bS_{\nu}(E)][\bS_{\nu'}(E')].
\end{displaymath}
As $[A]=\sum_{\mu} [S_{\mu}(E)][S_{\mu}(E')]$, we obtain
\begin{displaymath}
[M_{\lambda,\lambda'}]=\sum_{\substack{\mu, \theta, \theta', \nu, \nu' \\ \tau_n(\nu,\nu')=(\lambda, \lambda') }} (-1)^{i_n(\nu, \nu')} c^{\theta}_{\nu, \mu} c^{\theta'}_{\nu', \mu} [\bS_{\theta}(E)] [\bS_{\theta'}(E')].
\end{displaymath}
We therefore find that the $\bS_{\theta}(E) \otimes \bS_{\theta'}(E')$ component of $M$ is given by
\begin{displaymath}
\sum_{\substack{\lambda, \lambda', \mu, \nu, \nu'\\ \tau_n(\nu,\nu')=(\lambda,\lambda')}} (-1)^{i_n(\nu,\nu')} c^{\theta}_{\nu, \mu} c^{\theta'}_{\nu', \mu} [\bS_{[\lambda,\lambda']}(V)].
\end{displaymath}
The result now follows.
\end{proof}

\begin{proposition}
Let $(\lambda, \lambda')$ be an admissible pair.  Then we have an isomorphism $B_{\lambda,\lambda'} \to M_{\lambda,\lambda'}$ which is $A$-linear and $\GL(E) \times \GL(E')$ equivariant.
\end{proposition}

\begin{proof}
As in the proof of Lemma~\ref{gl:m-ind}, we have a presentation
\begin{displaymath}
A \otimes \bS_{(\lambda, 1^d)}(E) \otimes \bS_{(\lambda', 1^d)}(E') \to A \otimes \bS_{\lambda}(E) \otimes \bS_{\lambda'}(E') \to M_{\lambda,\lambda'} \to 0,
\end{displaymath}
where $d=n+1-\ell(\lambda)-\ell(\lambda')$.  Note that $\bS_{(\lambda,1^d)}(E) \otimes \bS_{(\lambda',1^d)}$ occurs with multiplicity one in the middle module, and thus does not occur in $M_{\lambda,\lambda'}$; it therefore does not occur in $B_{\lambda,\lambda'}$ either, since $M_{\lambda,\lambda'}$ and $B_{\lambda,\lambda'}$ are isomorphic as representations of $\GL(E) \times \GL(E')$.

Since $\Tor^A_0(B, \bC)=C$, we see from \eqref{gl:cring} that $\Tor^A_0(B_{\lambda,\lambda'}, \bC)=\bS_{\lambda}(E) \otimes \bS_{\lambda'}(E')$.  It follows that we have a surjection $f \colon A \otimes \bS_{\lambda}(E) \otimes \bS_{\lambda'}(E') \to B_{\lambda, \lambda'}$.  Since $\bS_{(\lambda, 1^d)}(E) \otimes \bS_{(\lambda', 1^d)}(E')$ does not occur in $B_{\lambda, \lambda'}$, we see that $f$ induces a surjection $M_{\lambda, \lambda'} \to B_{\lambda, \lambda'}$.  Since the two are isomorphic as $\GL(E) \times \GL(E')$ representations and have finite multiplicities, this surjection is an isomorphism.
\end{proof}

Combining the above proposition with Proposition~\ref{gl:stepb}, we obtain the following corollary.  Combined with \eqref{gl:tor}, this proves the main theorem.

\begin{corollary}
We have
\begin{displaymath}
\Tor^A_i(B, \bC)=\bigoplus_{i_n(\lambda,\lambda')=i} \bS_{\lambda}(E) \otimes \bS_{\lambda'}(E') \otimes \bS_{[\tau_n(\lambda, \lambda')]}(V).
\end{displaymath}
\end{corollary}

\subsection{Examples}
\label{ss:genlinexamples}

We now give a few examples to illustrate the theorem.  

\begin{example}
Suppose $\lambda=(1^i)$, $\lambda'=(1^j)$ with $i$ and $j$ positive.  Then $L^{\lambda ,\lambda'}_\bullet$ is the complex
\begin{displaymath}
\lw^{i-1}V\otimes\lw^{j-1}V^*\rightarrow \lw^iV\otimes\lw^j V^*,
\end{displaymath}
where the differential is the multiplication by the identity, treated as an element of $V\otimes V^*$. Let $n = \dim V$.
\begin{itemize}
\item If $i+j\le n$ then the differential is injective, and $\rH_0(L^{\lambda,\lambda'}_\bullet)=\bS_{[1^i, 1^j]}(V)$ is the irreducible representation with highest weight $(1^i, 0^{n-i-j}, (-1)^j)$.  
\item If $i+j=n+1$ then the differential is an isomorphism, and all homology of $L^{\lambda ,\lambda'}_\bullet$ vanishes.  
\item If $i+j>n+1$ but $i \le n+1$ and $j \le n+1$ we have $\rH_1(L^{\lambda, \lambda'}_\bullet)= \bS_{[1^{n+1-j},1^{n+1-i}]}(V)$.  
\item Finally, if $i>n+1$ or $j>n+1$ then the complex $L^{\lambda, \lambda'}_\bullet$ vanishes identically.
\end{itemize}

The reader will easily check that the description of homology given above agrees with the rule given by the Weyl group action.
\end{example}

\begin{example}
Suppose $\lambda=\lambda_0=(4,3,2,2)$ and $\lambda'=\lambda'_0=(5,2,2,1,1)$ and $n=3$.  We are supposed to remove border strips of size $\ell(\lambda_0)+\ell(\lambda'_0)-n-1=4+5-4=5$ from each partition.  Let $R_0$ and $R_0'$ be these border strips.  The picture is as follows:
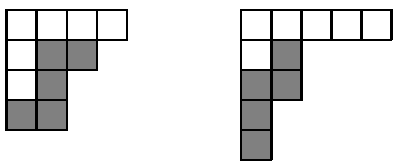
\begin{displaymath}
\ydiagram[*(white)]{4,1,1}*[*(gray)]{4,3,2,2} \qquad\qquad
\ydiagram[*(white)]{5,1}*[*(gray)]{5,2,2,1,1}
\end{displaymath}
The partition $\lambda_0$ is on the left, with $R_0$ shaded, and $\lambda'_0$ is on the right with $R'_0$ shaded.  Let $\lambda_1=\lambda_0 \setminus R_0$ and $\lambda_1'=\lambda_0' \setminus R_0'$; these are the unshaded boxes in the above diagrams.  As $\ell(\lambda_1)+\ell(\lambda_1')=5>n$, the pair $(\lambda, \lambda')$ is not admissible and the algorithm continues.  The border strips $R_1$ and $R_1'$ have size 1.  The picture is thus:
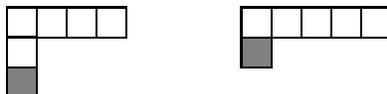
\begin{displaymath}
\ydiagram[*(white)]{4,1}*[*(gray)]{4,1,1} \qquad\qquad
\ydiagram[*(white)]{5}*[*(gray)]{5,1}
\end{displaymath}
Removing these border strips, we obtain the partitions $\lambda_2=(4,1)$ and $\lambda_2'=(4)$.  As $\ell(\lambda_2)+\ell(\lambda_2') \le n$, the pair $(\lambda, \lambda')$  is admissible and the algorithm terminates.  So $\tau_3(\lambda, \lambda')=((4,1), (4))$ and
\begin{displaymath}
i_3(\lambda, \lambda')=(c(R_0)+c(R_0')-1)+(c(R_1)+c(R_1')-1)=4+1=5.
\end{displaymath}
It follows that $\rH_i(L^{\lambda,\lambda'}_\bullet)=0$ for $i \ne 5$, while $\rH_5(L^{\lambda,\lambda'}_\bullet)=\bS_{[(4,1),(4)]}(V)$ is the irreducible of $\GL(3)$ with highest weight $(4,1,-4)$.

Now we illustrate the modification rule using Koike's original Weyl group action. Using the notation of \S\ref{ss:gl-modrule}, we have $\sigma(\lambda') = (2,2,2,0,-2)$, $\alpha = (2,2,2,0,-2 \mid 4,4,2,1)$, and $\rho = (\dots, 2,1 \mid 0,-1,\dots)$. If we sort $\alpha + \rho = (7,6,5,2,-1 \mid 4,3,0,-2)$, we get $(7,6,5,4,3 \mid 2,0,-1,-2)$. The permutation that does this sorting has length 5 (we made 5 consecutive swaps), and subtracting $\rho$, we get $(2,2,2,2,2 \mid 2,1,1,1)$. Then $(2,1,1,1)^\dagger = (4,1)$ is our first partition, and our second partition is $\sigma^{-1}(2^5) = (1^5)^\dagger = (5)$.
\end{example}


\begin{thebibliography}{HTW}

\bibitem[AH]{ah} Luchezar Avramov, J\"urgen Herzog, The Koszul algebra of a codimension 2 embedding, {\it Math. Z.} {\bf 175} (1980), no.~3, 249--260.

\bibitem[Bry]{brylinski} Ranee Kathryn Brylinski, Matrix concomitants with the mixed tensor model, {\it Adv. Math.} {\bf 100} (1993), no.~1, 28--52. 

\bibitem[DPS]{categorytg} Elizabeth Dan-Cohen, Ivan Penkov, Vera Serganova, A Koszul category of representations of finitary Lie algebras, \arxiv{1105.3407v2}.

\bibitem[EW]{ew} Thomas J. Enright, Jeb F. Willenbring, Hilbert series, Howe duality and branching for classical groups, {\it Ann. of Math. (2)} {\bf 159} (2004), no.~1, 337--375.

\bibitem[FH]{fultonharris} William Fulton, Joe Harris, {\it Representation Theory: A First Course}, Graduate Texts in Mathematics {\bf 129}, Springer-Verlag, New York, 1991.

\bibitem[How]{howe} Roger Howe, Perspectives on invariant theory: Schur duality, multiplicity-free actions and beyond, {\it Israel Mathematical Conference Proceedings} {\bf 8}, 1995.

\bibitem[HTW]{htw} Roger Howe, Eng-Chye Tan, Jeb F. Willenbring, Stable branching rules for classical symmetric pairs, {\it Trans. Amer. Math. Soc.} {\bf 357} (2005), no.~4, 1601--1626, \arxiv{math/0311159v2}.

\bibitem[Hum]{humphreys} James E. Humphreys, {\it Reflection Groups and Coxeter Groups}, Cambridge Studies in Advanced Mathematics {\bf 29}, Cambridge University Press, Cambridge, 1990. 

\bibitem[JPW]{jpw} T.~J\'ozefiak, P.~Pragacz, J.~Weyman, Resolutions of determinantal varieties and tensor complexes associated with symmetric and antisymmetric matrices, {\it Young tableaux and Schur functors in algebra and geometry (Toru\'n, 1980)}, pp. 109--189, Ast\'erisque, {\bf 87-88}, Soc. Math. France, Paris, 1981. 

\bibitem[Kin]{king} R.~C.~King, Modification rules and products of irreducible representations of the unitary, orthogonal, and symplectic groups, {\it J. Mathematical Phys.} {\bf 12} (1971), 1588--1598. 

\bibitem[Koi]{koike} Kazuhiko Koike, On the decomposition of tensor products of the representations of the classical groups: by means of the universal characters, {\it Adv. Math.} {\bf 74} (1989), no.~1, 57--86.

\bibitem[KT]{koiketerada} Kazuhiko Koike, Itaru Terada, Young-diagrammatic methods for the representation theory of the classical groups of type $B_n$, $C_n$, $D_n$, {\it J. Algebra} {\bf 107} (1987), no.~2, 466--511.

\bibitem[Las]{lascoux} Alain Lascoux, Syzygies des vari\'et\'es d\'eterminantales, {\it Adv. in Math.} {\bf 30} (1978), no.~3, 202--237. 

\bibitem[Lit]{littlewood} Dudley E. Littlewood, {\it The Theory of Group Characters and Matrix Representations of Groups}, reprint of the second (1950) edition, AMS Chelsea Publishing, Providence, RI, 2006.

\bibitem[Mac]{macdonald} I.~G.~Macdonald, {\it Symmetric Functions and Hall Polynomials}, second edition, Oxford Mathematical Monographs, Oxford, 1995.

\bibitem[SS1]{tca} Steven~V Sam, Andrew Snowden, GL-equivariant modules over polynomial rings in infinitely many variables, \arxiv{1206.2233v1}.

\bibitem[SS2]{algrep} Steven~V Sam, Andrew Snowden, Stability patterns in representation theory, in preparation.

\bibitem[SW]{koszulhomology} Steven~V Sam, Jerzy Weyman, Koszul homology of codimension 3 Gorenstein ideals, {\it Proc. Amer. Math. Soc.}, to appear, \arxiv{1203.3168v1}.

\bibitem[Sun]{sundaram} Sheila Sundaram, Tableaux in the representation theory of the classical Lie groups, {\it Invariant theory and tableaux (Minneapolis, MN, 1988)}, 191--225, IMA Vol. Math. Appl., {\bf 19}, Springer, New York, 1990.

\bibitem[Wen]{wenzl} Hans Wenzl, Quotients of representation rings, {\it Represent. Theory} {\bf 15} (2011), 385--406, \arxiv{1101.5887v1}.

\bibitem[Wey]{weyman} Jerzy Weyman, {\it Cohomology of Vector Bundles and Syzygies}, Cambridge University Press, Cambridge, 2003.

\end{thebibliography}
\end{document}